\documentclass[a4paper]{article}

\usepackage[all]{xy}\usepackage[latin1]{inputenc}        
\usepackage[dvips]{graphics,graphicx}
\usepackage{amsfonts,amssymb,amsmath,color,mathrsfs, amstext}
\usepackage{amsbsy, amsopn, amscd, amsxtra, amsthm,authblk,enumerate}
\usepackage{enumitem}
\usepackage{multirow}
\usepackage{upref}
\usepackage{makecell}
\usepackage[toc,page]{appendix}
\usepackage[colorlinks,
linkcolor=red,
anchorcolor=red,
citecolor=red
]{hyperref}
\usepackage{calc}

\usepackage{geometry}
\usepackage{float}
\usepackage[pagewise]{lineno}

\usepackage{yhmath}
\DeclareMathOperator{\Res}{Res}

\DeclareMathOperator{\p}{p.v.}

\DeclareSymbolFont{largesymbol}{OMX}{yhex}{m}{n}
\DeclareMathAccent{\Widehat}{\mathord}{largesymbol}{"62}
\newcommand*\di{\mathop{}\!\mathrm{d}}

\def\e{\epsilon}

\newcommand{\occ}{\overline{\mathbb{C}_+}}

\numberwithin{equation}{section}              
\newtheorem{theorem}{Theorem}[section]
\newtheorem{lemma}{Lemma}[section]
\newtheorem{proposition}{Proposition}[section]
\newtheorem*{proposition*}{Proposition}
\newtheorem{corollary}{Corollary}[section]
\newtheorem*{corollary*}{Corollary}
\newtheorem{definition}{Definition}[section]
\newtheorem*{definitions*}{Definitions}
\newtheorem*{conjecture*}{\bf Conjecture}

\newtheorem*{example*}{\bf Example}
\theoremstyle{remark}
\newtheorem{remark}{\bf Remark}[section]

\begin{document}
\date{}                                     
\title{Large time behavior, bi-Hamiltonian structure and kinetic formulation for complex Burgers equation}

\author[1]{Yu Gao\thanks{gaoyu90@hku.hk}}
\author[2]{Yuan Gao\thanks{yuangao@math.duke.edu}}
\author[3]{Jian-Guo Liu\thanks{jliu@phy.duke.edu}}
\affil[1]{Department of Mathematics, The University of Hong Kong, Pokfulam, Hong Kong.}
\affil[2]{Department of Mathematics, Duke University, Durham, NC 27708, USA.}
\affil[3]{Department of Mathematics and Department of Physics, Duke University, Durham, NC 27708, USA.}

\maketitle

\begin{abstract}
We prove the  existence and uniqueness of positive analytical solutions with positive initial data to the mean field equation (the Dyson equation) of the Dyson Brownian motion  through the complex Burgers equation with a force term on the upper half complex plane. These solutions converge to a steady state given by Wigner's semicircle law. A unique global weak solution with nonnegative initial data to the Dyson equation is obtained and some explicit solutions are given by Wigner's semicircle laws. We also construct a bi-Hamiltonian structure for the system of the real and imaginary components of the complex Burgers equation (coupled Burgers system). We establish a kinetic formulation for the coupled Burgers system and prove the existence and uniqueness of entropy solutions. 
The coupled Burgers system in Lagrangian variable naturally leads to two interacting particle systems: Fermi-Pasta-Ulam-Tsingou model with nearest-neighbor interactions, and Calogero-Moser model.  These two particle systems yield the same Lagrangian dynamics in the continuum limit.

\end{abstract}

\section{Introduction}
Complex Burgers equation arises, although in different ways, from  many different fields such as fluid mechanics,  random surface minimizing problem and Burgers turbulence in quantum chromodynamics, which always unveils some mechanisms of singularity formations. We only list several examples here.
\cite{zubarev2018exact} use complex Burgers equation to construct a family of  singular solution to zero-gravity water wave system.
\cite{kenyon2007limit} use the complex Burgers equation to study the limit shape and singularity formations of  random surface models.
For other applications of complex Burgers equation such as singularity tracking in the evolution of the complex system and the large-N limit of induced quantum chromodynamics   we refer to \cite{gross1995some} and the references therein.

In this paper, we study the complex Burgers equation with a force term $\gamma^2z$ on the upper half complex plane $\mathbb{C}_+:=\{z:\Im(z)>0\}$:
\begin{align}\label{eq:complexBurgers}
\partial_tg+g\partial_zg=\gamma^2z,\qquad z\in\mathbb{C_+},~~t>0.
\end{align}
Here, $\gamma\geq0$ is a constant. We use $\Re(z)$ and $\Im(z)$ to stand for the real and imaginary parts of a complex number $z$ respectively.

Take the trace of a solution $g(z,t)$ to \eqref{eq:complexBurgers} on the real line and there are two real functions $u(x,t)$ and $\rho(x,t)$ such that
\begin{align}\label{eq:furho}
g(x,t)+\gamma x=u(x,t)+i\pi \rho(x,t),~~x\in\mathbb{R},~~t>0,
\end{align}
where $\pi$ is the circumference ratio.
If  $g(z,t)$ is a $\mathbb{C}_+$-holomorphic function, then we have the following relation between $u$ and $\rho$:
\begin{align}\label{eq:hilbert}
u(x,t)=(\pi H\rho) (x,t),
\end{align}
where $H\rho$ stands for the Hilbert transform of $\rho$ given by 
\begin{align*}
(H\rho)(x,t)=\frac{1}{\pi}\p\int_{\mathbb{R}}\frac{\rho(y,t)}{x-y}\di y,~~x\in\mathbb{R}.
\end{align*}
Take \eqref{eq:furho} into \eqref{eq:complexBurgers} and we obtain the following nonlocal partial differential equation for $\rho$: 
\begin{equation}\label{eq:meanfield}
\partial_t\rho+ \partial_x[\rho (u-\gamma x)]=0,~~u(x,t)=(\pi H\rho)(x,t),~~x\in\mathbb{R},~~t>0.
\end{equation}
The equation for $u$ can be obtained from \eqref{eq:meanfield} by the Hilbert transform (see  \eqref{eq:ubeta1}).
We refer to  \eqref{eq:meanfield} as the Dyson equation which is a mean field equation for the Dyson Brownian motion as described below.

The $N\times N$ complex Hermitian matrices form a $N^2$ dimensional linear vector space over field $\mathbb{R}$. Consider a Hermitian matrices valued Ornstein-Uhlenbeck (OU) process $A(t)=(A_{jk}(t))_{N\times N}$ given by
\begin{gather}\label{eq:OUP}
\left\{\begin{split}
\di A_{jj}(t)&=\frac{1}{\sqrt{N}}\di B_{jj}(t)-\gamma A_{jj}(t) \di t,~~j=1,\cdots, N,\\
\di \Re A_{jk}(t)&=\frac{1}{\sqrt{2N}}\di B^R_{jk}(t)-\gamma \Re A_{jk}(t) \di t,~~j< k,\\
\di \Im A_{jk}(t)&=\frac{1}{\sqrt{2N}}\di B^I_{jk}(t)-\gamma \Im A_{jk}(t) \di t,~~j<k,
\end{split}
\right.
\end{gather}
with $A(0)=0.$
Here $B_{jj}(t)~~(1\leq j\leq N),~~B^R_{jk}(t),~B^I_{jk}(t)~~(1\leq j<k\leq N),$ are $N^2$ independent standard Brownian motions in $\mathbb{R}$.  The eigenvalues $\lambda_1(t)\leq \cdots\leq \lambda_N(t)$ of $A(t)$ form some real stochastic processes.
By applying Ito's formula to $\lambda_j(t)(=\lambda_j(A(t)))$, one can show that $\lambda_j(t)$ evolve by  (\cite{dyson1962brownian,erdos2017dynamical,Tao2012Topics})
\begin{align}\label{eq:DysonBrownian}
\di \lambda_j(t)=\frac{1}{\sqrt{N}}\di B_j(t)+\frac{1}{N}\sum_{k\neq j}\frac{\di t}{\lambda_j(t)-\lambda_k(t)}-\gamma\lambda_j(t)\di t,\quad 1\leq j\leq N.
\end{align}
This evolution of eigenvalues are referred to as the Dyson Brownian motion. One can refer to \cite{erdos2017dynamical,Tao2012Topics} for more details about random matrices and the Dyson Brownian motion. It is well known that the effects of harmonic trap term $-\gamma \lambda_j(t)\di t$ in the OU process \eqref{eq:DysonBrownian} can be reformulated into the case $\gamma=0$, i.e. \eqref{eq:DysonBrownian} without the trap term, by a space-time rescaling. We describe this space-time rescaling for complex Burgers equation below. Let $g$ be a solution to \eqref{eq:complexBurgers} and set
\begin{align}\label{eq:transform}
\tilde{g}(w,\tau)\sqrt{1+2\gamma\tau}=g(z,t)+\gamma z,\quad z=\frac{w}{\sqrt{1+2\gamma\tau}},\quad t=\frac{1}{2\gamma}\log(1+2\gamma\tau).
\end{align}
Then, $\tilde{g}$ is a solution to the Complex Burgers equation without the force term:
\begin{align}\label{eq:complexB}
\partial_\tau\tilde{g}+\tilde{g}\partial_w\tilde{g}=0.
\end{align}
Note that $\tilde{g}(\cdot,\tau)$ is a $\mathbb{C}_+$-holomorphic ($\occ$-holomorphic) solution to \eqref{eq:complexB}  if and only if $g(\cdot,t)$ is a $\mathbb{C}_+$-holomorphic ($\occ$-holomorphic) solution to \eqref{eq:complexBurgers}.

The mean field limit of the Dyson Brownian motion \eqref{eq:DysonBrownian} yields the Dyson equation \eqref{eq:meanfield} (\cite{rogers1993interacting,cepa1997diffusing,berman2019propagation}), and \eqref{eq:meanfield} is a gradient flow in the probability measure spaces with Wasserstein distance  with respect to a free energy functional given by \cite[Chapter 11]{Ambrosio}
\begin{align}\label{eq:interactionEnergy}
E(\rho(\cdot,t))&=\frac{\gamma}{2}\int_{\mathbb{R}} x^2\rho(x,t) \di x-\frac{1}{2}\int_{\mathbb{R}}\int_{\mathbb{R}}\log|x-y|\rho(x,t)\rho(y,t)\di x\di y\nonumber\\
&=:E_{\textnormal{h}}(\rho(\cdot,t))+E_{\textnormal{i}}(\rho(\cdot,t)).
\end{align}
Here $E_{\textnormal{h}}$ is a harmonic  trap  energy  and  $E_{\textnormal{i}}$ is  an interaction energy.
Then, the  Dyson equation \eqref{eq:meanfield} is recast to
\begin{align}\label{eq:gradientflow}
\partial_t\rho-\partial_x\left[\rho\partial_x\left(\frac{\delta E}{\delta\rho}\right)\right]=0,\quad \frac{\delta E}{\delta\rho}=  \frac{\gamma}{2} x^2-\int_{\mathbb{R}}\log|x-y|\rho(y,t)\di y.
\end{align}

With initial data $\rho_0>0$ and $\rho_0\in L^2(\mathbb{R})\cap C^{0,\delta}(\mathbb{R})$, Castro and C\'ordoba \cite{Castro2008Global} proved global existence and uniqueness of  real analytical solutions for $t>0$ to the case  $\gamma=0$ of \eqref{eq:meanfield} . This instantaneous analytical property is suggested by the gradient flow structure \eqref{eq:gradientflow}. However, if there is $x_0\in\mathbb{R}$ such that  $\rho(x_0)=0$, then the solution $ \rho$ will blow up in $H^s(\mathbb{R}),\, s>\frac{3}{2}$ at finite time \cite{Castro2008Global}.  Thanks to the transformation in \eqref{eq:transform},  these two results hold also for $\gamma>0$; see Theorem \ref{thm:analytic} and Remark \ref{eq:blowup}. Moreover, we prove the global weak solution $\rho\in L^{\infty}(0,T; H^{\frac12}(\mathbb{R})\cap L^1_+(\mathbb{R}))$ to \eqref{eq:meanfield} in Theorem \ref{thm:weaktheorem}. The global regularity or finite time blow in the space $H^s(\mathbb{R}),\,s\in(\frac12, \frac32]$ remain open.

The steady state for the Dyson equation is given by Wigner's semicircle law:
\begin{align}\label{eq:semicirclelaw}
\mu_1(\di x)= \rho_1(x) \di x := \frac{\sqrt{(4-x^2)_+}}{2\pi}\di x,
\end{align}
which has a compact support. Hence the solution $\rho$ is not absolutely continuous with respect to the steady state and the relative entropy method can not be directly applied here. There are two methods to prove the convergence of solution $\rho$ to its steady state. (i) For strictly positive initial data $\rho_0(x)>0$, following the idea of \cite{rogers1993interacting} we prove the pointwise convergence as $t$ goes to infinity using analytical method; see Appendix \ref{App_B}. (ii) Notice  the free energy $E(\rho)$ given by \eqref{eq:interactionEnergy} for the Dyson equation consists a harmonic  trap  energy $E_{\textnormal{h}}$ and an interaction energy $E_{\textnormal{i}}$. Since $E_{\textnormal{i}}$ is  convex along generalized Wasserstein geodesics and $E_{\textnormal{h}}$ is $\gamma$-convex along Wasserstein geodesics, the standard gradient flow theory yields $W_2$-contraction and hence the exponentially convergence to the steady state in Wasserstein distance (see Remark \ref{rmk:convergence} and Carrillo et. al. \cite{carrillo2012mass}). 

Consider the complex Burgers equation \eqref{eq:complexBurgers} with $\gamma=0$. If $g(x,t)$ given by  \eqref{eq:furho}  is no longer a trace of a $\mathbb{C}_+$-holomorphic function, then the relation between $u$ and $\rho$ in \eqref{eq:hilbert} does not hold. We need to treat $u$ and $\rho$ independently. Take \eqref{eq:furho} into \eqref{eq:complexBurgers} and we obtain the following system on the real line:
\begin{gather}\label{eq:systemEuler}
\left\{
\begin{split}
&\rho_t + (\rho u)_x = 0,~~x\in\mathbb{R},~~t>0,\\
&\partial_tu + u \partial_xu-\pi^2\rho\partial_x\rho = 0.
\end{split}
\right.
\end{gather}
Unfortunately,  for the Cauchy problem, the above system is ill-posed  as described below. We introduce the following system of conservation law with general constant $\alpha\in\mathbb{R}$:
\begin{gather}\label{eq:systemEuler1}
\left\{
\begin{split}
&\partial_t\rho + \partial_x(\rho u) = 0,~~x\in\mathbb{R},~~t>0,\\
&\partial_tu + \partial_x\left(\frac{u^2+\alpha \rho^2}{2}\right)= 0.
\end{split}
\right.
\end{gather}
Due to the relation between System \eqref{eq:systemEuler} ( \eqref{eq:systemEuler1}) and the complex Burgers equation \eqref{eq:complexBurgers}, we call System \eqref{eq:systemEuler1} as the coupled Burgers system in this paper. System \eqref{eq:systemEuler1} can be rewritten as the following quasi-linear system
\begin{equation}\label{eq:SystemEuler}
\frac{\partial}{\partial t}\begin{pmatrix}
\rho\\
u
\end{pmatrix}
+
A(\rho,u)\frac{\partial}{\partial x}
\begin{pmatrix}
\rho \\  u
\end{pmatrix}=0,\quad A(\rho,u)
=
\begin{pmatrix}
u & \rho\\
\alpha\rho  & u
\end{pmatrix}.
\end{equation}
The eigenvalues of $A$ are given by $u\pm \sqrt{\alpha}\rho$, where $\sqrt{\alpha}=\sqrt{-1}\sqrt{|\alpha|}=i\sqrt{|\alpha|}$ for $\alpha<0$. When $\alpha>0$, this system is a hyperbolic system of conservation laws. When $\alpha<0$ and $\rho\neq 0$, $A$ has two imaginary eigenvalues and System \eqref{eq:systemEuler1} is elliptic and ill-posedness. For  $\alpha\neq 0$, we set the eigenvalues as
\begin{equation}\label{fpm}
f_+:=u+\sqrt{\alpha}\rho,\quad f_-:=u-\sqrt{\alpha}\rho.
\end{equation}
A linear transformation from the coupled Burgers system \eqref{eq:systemEuler1} shows that the eigenvalues satisfy the following decoupled Burgers equations:
\begin{align}
&\partial_tf_++f_+\partial_xf_+=0,\quad x\in\mathbb{R},~~t>0,\label{eq:decouple1}\\
&\partial_tf_-+f_-\partial_xf_-=0,\quad x\in\mathbb{R},~~t>0.\label{eq:decouple2}
\end{align}
When $\alpha<0$, \eqref{eq:decouple2} is just the conjugate of equation  \eqref{eq:decouple1}. When $\alpha=-\pi^2$, \eqref{eq:decouple1} is exactly the complex Burgers equation \eqref{eq:complexBurgers} ($\gamma=0$) on the real line.

For $\alpha>1$, notice that $f_\pm$ are Riemann invariants of the following system of  isentropic gas dynamics:
\begin{gather}\label{eq:systemEuler3}
\left\{
\begin{split}
& \partial_t\rho + \partial_x(\rho u) = 0,~~x\in\mathbb{R},~~t>0,\\
&\partial_t(\rho u) + \partial_x(\rho u^2) + \partial_xp = 0 ,
\end{split}
\right.
\end{gather}
where the pressure $p$ is given by
\begin{align}\label{eq:pressure}
p(x,t)=\frac{\alpha}{3}\rho^3(x,t).
\end{align}
Formally, system  \eqref{eq:systemEuler3} is a nonlinear transformation of the coupled Burgers system \eqref{eq:systemEuler1} and it expresses in physics the conservation of mass and the conservation of momentum, i.e. $m:=\rho u$, for an  isentropic gas system. In the quasi-linear form, we have
\begin{equation}\label{eq:SystemEulergas}
\frac{\partial}{\partial t}\begin{pmatrix}
\rho\\
m
\end{pmatrix}
+
B(\rho,m)\frac{\partial}{\partial x}
\begin{pmatrix}
\rho \\  m
\end{pmatrix}=0,\quad B(\rho,m)
=
\begin{pmatrix}
0& 1 \\
-\frac{m^2}{\rho^2}+\alpha\rho^2 & \frac{2m}{\rho} 
\end{pmatrix}.
\end{equation}
The functions $f_\pm=u\pm\sqrt{\alpha}\rho$ are also the eigenvalues of $B$. Notice that classical solutions of  the coupled Burgers system \eqref{eq:systemEuler1} are also classical solutions to \eqref{eq:systemEuler3}. However, when shock appears, shock speed for the coupled Burgers system \eqref{eq:systemEuler1} and \eqref{eq:systemEuler3} are different.
For smooth solutions of System \eqref{eq:systemEuler3}, the following conservation of energy  holds:
\begin{align}\label{eq:energyconservated}
\partial_tE + \partial_x[u(E+ p)]=0,
\end{align}
where the total energy density is given by
\begin{align}\label{eq:energyflux}
E(x,t) = \frac{1}{2} \rho u^2 + \frac{p}{2} =\frac{1}{2} \rho u^2 + \frac{\alpha }{6}\rho^3. 
\end{align}

Although there is no bi-Hamiltonian structure for Burgers equation, we use the decoupled Burgers equations \eqref{eq:decouple1} and \eqref{eq:decouple2} to construct a bi-Hamiltonian structure for the coupled Burgers system \eqref{eq:systemEuler1} (see Theorem \ref{thm:biHamiltonian}). Moreover, we obtain infinite many conserved quantities for the coupled Burgers system \eqref{eq:systemEuler1}. Bi-Hamiltonian structures for System \eqref{eq:systemEuler3} and p-system (which is the gas dynamics in Lagrangian coordinates; see  \eqref{eq:gas} below) are also obtained. To discover a bi-Hamiltonian structure or a Lax pair for an integrable system is very important. Indeed, according to the fundamental theorem of Magri  \cite{Magri1978A}, any bi-Hamiltonian system associated with a nondegenerate Hamiltonian pair induces a hierarchy of commuting Hamiltonian flows and, provided enough of these Hamiltonians are functionally independent, is therefore completely integrable. For general discussions about Hamiltonian structures for systems of hyperbolic conservation laws, one can refer to \cite{Olver1988Hamiltonian}.

When $\alpha>0$, we establish a kinetic formulation for the coupled Burgers system \eqref{eq:systemEuler1}. Using the kinetic formulation, we define a class of entropy pairs to the coupled Burgers system \eqref{eq:systemEuler1}. Notice that our definition of entropies corresponds to the counter part (in the sense as explained in Remark \ref{rmk:counterpart}) of entropies used in \cite{Lions1994Kinetic} for System \eqref{eq:systemEuler3}. In \cite{Lions1994Kinetic}, Lions, Perthame and Tadmor proved the existence of global entropy solutions to \eqref{eq:systemEuler3} and the uniqueness is unknown. In contrast, we prove the existence and uniqueness of entropy solutions to the coupled Burgers system \eqref{eq:systemEuler1} (see  Section \ref{sec:entropy}). Moreover, we show that an entropy solution to the coupled Burgers system \eqref{eq:systemEuler1}  corresponds to an entropy solution to the decoupled Burgers equations \eqref{eq:decouple1} and \eqref{eq:decouple2} (see Proposition \ref{pro:equivalence}). For more details on relations of  entropy solutions and weak solutions to kinetic equations, one can refer to \cite{Perthame2000Kinetic}.

We also derive the Lagrangian dynamics (see \eqref{eq:LagrangeEuler}) for the coupled Burgers system \eqref{eq:systemEuler1}, which resembles the gas dynamics in Lagrangian variables, or p-system \cite{Serre1999Systems}:
\begin{gather}\label{eq:gas}
\left\{
\begin{split}
&\partial_t\tau- \partial_\xi V=0,\\
&\partial_t V+\partial_\xi p=0,
\end{split}
\right.
\end{gather}
where $\tau(\xi,t)={1}/{\rho(X(\xi,t),t)}=X_\xi(\xi,t)$ stands for the specific volume and $\xi$ is the Lagrangian labels. $X(\xi,t)$ is the flow map according to velocity field $u(X(\xi,t),t)$ (see \eqref{eq:leastdynamics}). $V$ is the velocity in Lagrangian variable $V(\xi,t):=u(X(\xi,t),t)$ and $p(\tau)={\alpha}/{(3\tau^3)}$ is the pressure given by \eqref{eq:pressure} (see more details in Section \ref{sec:Lagrange}). The Lagrangian dynamics of the coupled Burgers system \eqref{eq:systemEuler1} naturally leads to a spring-mass system (Fermi-Pasta-Ulam-Tsingou model) such that each mass evolves by the elastic force between adjacent mass that are reciprocal proportion to the cubic of distances between the mass and the adjacent masses (see \eqref{eq:springModel}). Instead of the nearest-neighbor interaction, if the mass interacts with all the other masses with the same manner, we obtain the Calogero-Moser model with different coefficients. As it is known, the Calogero-Moser model is an integrable systems with a Lax-pair; see \cite{moser1976three}. An interesting fact is that the continuum limit of the Calogero-Moser model gives the same Lagrangian dynamics of the coupled Burgers system \eqref{eq:systemEuler1}; see \cite{Menon2}.

The rest of this paper is organized as follows.
In Section \ref{sec:Dyson}, we prove the global existence and uniqueness of real analytical solutions to complex Burgers equation \eqref{eq:complexBurgers} and the Dyson equation \eqref{eq:meanfield} ($\gamma\geq 0$) with strictly positive initial datum $\rho_0\in H^s(\mathbb{R})\cap L^1(\mathbb{R})$, $s>1/2$. We also obtian the pointwise convergence to the steady state for analytical solutions. Some explicit solutions are constructed by using Wigner's semicircle law, which converge to the steady state exponentially when $\gamma>0$. The same explicit solution is given in Appendix \ref{app_A} by the Stieltjes transform of Wigner's semicircle law $\mu_1$. Moreover, we prove the global existence of weak solutions in $H^{1/2}(\mathbb{R})\cap L^1(\mathbb{R})$ for nonnegative initial date.
In Section \ref{sec:bihamiltonian}, we construct bi-Hamiltonian structures for the coupled Burgers system \eqref{eq:systemEuler1}, isentropic gas system \eqref{eq:systemEuler3} and p-system \eqref{eq:gas}.
In Section \ref{sec:kinetic}, we establish kinetic formulation for the  coupled Burgers system \eqref{eq:systemEuler1} with $\alpha>0$. The existence and uniqueness of entropy solutions to  \eqref{eq:systemEuler1} are also proved. In Section \ref{sec:Lagrangemass}, we study the Lagrangian dynamics for the coupled Burgers system \eqref{eq:systemEuler1} and explore the connection between the Lagrangian dynamics system and a Fermi-Pasta-Ulam-Tsingou model with nearest-neighbor interactions.  In Appendix \ref{App_B}, we give the proof of Theorem \ref{thm:analytic}.

\section{Complex Burgers equation and the Dyson Brownian motion}\label{sec:Dyson}
Recall the Dyson Brownian motion  \eqref{eq:DysonBrownian}. The eigenvalues $\lambda_j$ given by  \eqref{eq:DysonBrownian} evolve by Brownian motion, combined with a deterministic repulsion force that repels nearby eigenvalues from each other with a strength inversely proportional to the separation.
Notice that System \eqref{eq:DysonBrownian}  can also be rewritten as
\begin{align}\label{eq:normDysonmotion2}
\di \lambda_j(t)=\frac{1}{\sqrt{{N}}}\di B_j(t)-\partial_{\lambda_j}\Phi(\lambda_1(t),\cdots,\lambda_N(t)),\quad 1\leq j\leq N,
\end{align}
with potential function given by
\begin{align}\label{eq:potential}
\Phi(\lambda_1(t),\cdots,\lambda_N(t)):=\frac{\gamma}{2}\sum_{j=1}^N\lambda_j^2(t)-\frac{1}{2N} \sum_{j=1}^N\sum_{k\neq j} \log|\lambda_j(t) - \lambda_k(t)|.
\end{align}
It can be proved that the eigenvalues almost surely not collide with each other (see  \cite{rogers1993interacting,liu2016propagation, LLY2019}) and  the solutions to System \eqref{eq:DysonBrownian} exist globally. Hence, the empirical measure
\begin{align}\label{eq:empirical}
\mu^N(t):=\frac{1}{N}\sum_{j=1}^N\delta_{\lambda_j(t)}
\end{align}
is well defined for $t\in[0,\infty)$. One can prove that $\mu^N(t)$ converges to some probability measure satisfying the Dyson equation  \eqref{eq:meanfield} (\cite{rogers1993interacting,cepa1997diffusing,berman2019propagation}).

Next, we derive the complex Burgers equation \eqref{eq:complexBurgers} from the Dyson equation \eqref{eq:meanfield}.
For $f,g\in L^{p}(\mathbb{R})$ ($p>1$), the Hilbert transform has the following properties (see e.g. \cite{Pandey}):
\[
H(Hf)=-f,\quad \partial_x(Hf)=H\partial_xf,
\]
and
\[
H(fHg+gHf)=HfHg-fg.
\]
Applying the Hilbert transform to the Dyson equation \eqref{eq:meanfield} yields
\[
\partial_t(H\rho) + \pi H\rho H\partial_x\rho - \pi\rho \partial_x\rho-\gamma\partial_xH(\rho x) = 0.
\]
Moreover, for any function $g:\mathbb{R}\to\mathbb{R}$, we have
\begin{align}\label{eq:formula}
H(xg(x))=&\frac{1}{\pi}\p\int_{\mathbb{R}}\frac{yg(y)}{x-y}\di y=\frac{1}{\pi}\p\int_{\mathbb{R}}\frac{(y-x)g(y)}{x-y}\di y+\frac{1}{\pi}\p\int_{\mathbb{R}}\frac{xg(y)}{x-y}\di y\nonumber\\
=&xHg(x)-\frac{1}{\pi}\int_{\mathbb{R}}g(x)\di x,
\end{align}
which implies
\begin{align}\label{fact2}
H(\rho x)=-\frac{\|\rho(t)\|_{L^1}}{\pi}+\frac{1}{\pi}ux.
\end{align}
Combining the above two equations, we have
\begin{align}\label{eq:ubeta1}
\partial_tu+u\partial_xu-\pi^2\rho \partial_x\rho -\gamma\partial_x(ux)= 0.
\end{align}
Set
\[
f = u - i \pi \rho,  \qquad u = \pi H \rho.
\]
Hence, $f$ gives the trace of an analytic function in the upper half plane. Combining \eqref{eq:meanfield} and \eqref{eq:ubeta1} yields
\[
\partial_t f+f\partial_xf-\gamma\partial_x(fx)=0,~~x\in\mathbb{R},~~t>0.
\]
This corresponds to the following complex equation in $\mathbb{C}_+$:
\begin{align}\label{eq:complexBurgers1}
\partial_t f+f\partial_zf-\gamma\partial_z(fz)=\partial_t f+f\partial_zf-\gamma z\partial_zf-\gamma f=0,~~t>0.
\end{align}
By the linear transformation $g(z,t)=f(z,t)-\gamma z$, we have
\[
\partial_tg+g\partial_zg-\gamma^2 z=\partial_tf+(f-\gamma z)(\partial_zf-\gamma)-\gamma^2z=\partial_t f+f\partial_zf-\gamma z\partial_zf-\gamma f=0,
\]
which is the Burgers equation with force term $\gamma^2 z$   \eqref{eq:complexBurgers}.
Moreover,  from the above computation we see that the Dyson equation  \eqref{eq:meanfield} with $\gamma=0$  is equivalent to the coupled Burgers system \eqref{eq:systemEuler1} with $\alpha=-\pi^2$ and $u=\pi H\rho$.

\subsection{Analytical solutions to the Dyson equation \eqref{eq:meanfield}, convergence to steady state and finite time blow up}\label{sec:analytic}

In this subsection, we prove the existence and uniqueness of positive analytical solutions to the Dyson equation \eqref{eq:meanfield} with $\gamma>0$ and initial datum $0<\rho_0\in H^s(\mathbb{R})\cap L^1(\mathbb{R})$ ($s>1/2$) by proving the well-posedness results for complex Burgers equation \eqref{eq:complexBurgers}. We also show the pointwise convergence to the steady state for analytical solutions. 

Let $\rho_0(x)>0$ and $\rho_0\in   H^s(\mathbb{R})\cap L^1(\mathbb{R})$ with $s>1/2$ be the initial datum for the Dyson equation \eqref{eq:meanfield}. The initial datum $\rho_0$  can be extended to a $\mathbb{C}_+$-holomorphic function by  Hilbert transform (also called Stieltjes transformation, Borel transform or Markov function) for positive measures: 
\begin{align}\label{eq:initialf}
f_0(z):=\frac{1}{\pi}\int_{\mathbb{R}}\frac{\rho_0(s)}{z-s}\di s,~~z=x+iy\in\mathbb{C}_+.
\end{align}
Let
\begin{align}\label{eq:initialg}
g_0(z):=f_0(z)-\gamma z,~~z=x+iy\in \mathbb{C}_+.
\end{align}
Then, $g_0$ is a $\mathbb{C_+}$-holomorphic function.
Consider the following Cauchy problem of the Burgers equation with force term $\gamma^2z$ in $\mathbb{C}_+$:
\begin{gather}\label{eq:complexBurgers2}
\left\{
\begin{split}
&[\partial_tg+g\partial_zg](z,t)=\gamma^2 z,~~~~z=x+iy\in \mathbb{C}_+,\\
&g(z,0)=g_0(z).
\end{split}
\right.
\end{gather}
First let us list some simple estimates
 for  the Dyson equation \eqref{eq:meanfield}.
\\Fact 1 ($L^1$-conservation law):
 $\|\rho(\cdot,t)\|_{L^1(\mathbb{R})}=\|\rho_0\|_{L^1(\mathbb{R})}$. 
 \\Fact 2 (Second moment estimate): Multiplying \eqref{eq:meanfield} by $x^2$ and taking integral yield
$$
\frac{\di}{\di t}\int_{\mathbb{R}}x^2\rho(x,t)\di x=2\pi\int_{\mathbb{R}}x\rho H\rho \di x-2\gamma \int_{\mathbb{R}}x^2\rho(x,t)\di x
$$
Notice from \eqref{fact2}, we have
$$
\int_{\mathbb{R}}x\rho H\rho \di x  = \frac{1}{2\pi}\|\rho\|_{L^1}^2,
$$
hence
$$
\frac{\di}{\di t}\int_{\mathbb{R}}x^2\rho(x,t)\di x=\|\rho_0\|_{L^1}^2-2\gamma\int_{\mathbb{R}}x^2\rho(x,t)\di x,
$$
which implies 
$$m_2(t)=\frac{\|\rho_0\|_{L^1}^2}{2\gamma}-\frac{\|\rho_0\|_{L^1}^2-2\gamma m_2(0)}{2\gamma}e^{-2\gamma t},~~\forall t>0.$$
\\Fact 3 ($L^2$ estimate): Multiplying \eqref{eq:meanfield} by $\rho$ and integration by parts show that
\begin{equation*}
\frac{\di }{\di t} \int_{\mathbb{R}}{\rho^2} \di x + 2\int_{\mathbb{R}} \int_{\mathbb{R}} \rho(x,t) \frac{|\rho(x,t)-\rho(y,t)|^2}{|x-y|^2} \di x\di y = \gamma\int_{\mathbb{R}}\rho^2\di x;
\end{equation*}
see more details in the proof of Theorem \ref{thm:weaktheorem}.
\\Fact 4 ($\dot{H}^{\frac12}$ estimate): 
\begin{equation*}
\frac{\di}{\di t} \|(-\Delta)^{1/4} \rho\|_{L^2}^2 + \pi\int_{\mathbb{R}}(\partial_xH\rho)^2 \rho\di x + \pi\int_{\mathbb{R}} \rho (\partial_x\rho)^2\di x=2\gamma\|(-\Delta)^{1/4}\rho\|_{L^2}^2;
\end{equation*}
see more details in the proof of Theorem \ref{thm:weaktheorem}.
\\Fact 5 (Entropy estimate): Taking the time derivative to $\int_\mathbb{R} \rho \log \rho \di x$ and integration by parts show that
\begin{align*}
\frac{\di}{\di t} \int_\mathbb{R} \rho \log \rho \di x =& \int_{\mathbb{R}} \partial_t \rho(\log \rho+1) \di x = \int_{\mathbb{R}} -(\rho H \rho+ \gamma x \rho)_x(\log \rho+1) \di x \\
=&  \int_{\mathbb{R}}(H \rho - \gamma x) \rho_x \di x  = - \|(-\Delta)^{1/4}\rho\|_{L^2}^2
\end{align*}
\\Fact 6 (Energy dissipation): Since the Dyson equation is a $W^2$-gradient flow with respect to the energy \eqref{eq:interactionEnergy},
we  have the following energy dissipation property 
\begin{align*}
\frac{\di}{\di t}E(\rho)&=\int_{\mathbb{R}}\frac{\delta E}{\delta \rho}\cdot \partial_t\rho \di x=-\int_{\mathbb{R}}\rho\left|\partial_x\left(\frac{\delta E}{\delta \rho}\right)\right|^2 \di x\nonumber\\
&=-\int_{\mathbb{R}}\rho(x,t)\big|\gamma x-\pi H\rho(x,t)\big|^2 \di x.
\end{align*}
Now we have the following theorem:
\begin{theorem}\label{thm:analytic}
Let  $\gamma\geq 0$ and $0<\rho_0\in H^s(\mathbb{R})\cap L^1(\mathbb{R})$ with $s>1/2$. Then,

$\mathrm{(i)}$ The complex Burgers equation \eqref{eq:complexBurgers2} has a unique $\occ$-holomorphic solution $g(\cdot,t)$ for $t\in(0,\infty)$, and $\frac{\partial^k}{\partial t^k}g(\cdot,t)$ is an analytical function of $z$ on $\occ$ for any positive integer $k$ and $t>0$.

$\mathrm{(ii)}$ For any $t>0$, the trace of $f(z,t)=g(z,t)+\gamma z$ on the real line gives a positive analytical solution $\rho(x,t)>0$ to the Dyson equation  \eqref{eq:meanfield} with $\rho(x,0)=\rho_0(x)$ and  $\frac{\partial^k}{\partial t^k}\rho(x,t)$ is an analytical function  of $x\in\mathbb{R}$ for any positive integer $k$. The following estimates hold:
\begin{enumerate}
\item[(a)] The total mass $\|\rho(t)\|_{L^1}$ is conserved:
\begin{align}\label{eq:L1conserved}
\|\rho(t)\|_{L^1}=\|\rho_0\|_{L^1}.
\end{align}
\item[(b)] If $x^2\rho_0\in L^1(\mathbb{R})$, then the second moment $m_2(t):=\int_{\mathbb{R}}x^2\rho(x,t)\di x$ satisfies
\begin{gather}\label{eq:secondmoment}
m_2(t)=\left\{
\begin{split}
&\frac{\|\rho_0\|_{L^1}^2}{2\gamma}-\frac{\|\rho_0\|_{L^1}^2-2\gamma m_2(0)}{2\gamma}e^{-2\gamma t},~~\gamma>0,\\
&m_2(0)+\|\rho_0\|_{L^1}^2t,~~\gamma=0.
\end{split}
\right.
\end{gather}
\item[(c)] The following energy dissipation holds:
\begin{align}\label{eq:energydissipation1}
\frac{\di}{\di t}E(\rho)=-\int_{\mathbb{R}}\rho(x,t)\big|\gamma x-\pi H\rho(x,t)\big|^2 \di x,
\end{align}
with $E$ defined by \eqref{eq:interactionEnergy}. 
\item[(d)] If $\rho_0\log\rho_0\in L^1(\mathbb{R})$, then the entropy $\theta(t):=\int_{\mathbb{R}}\rho(x,t)\log\rho(x,t)\di x$ satisfies
\begin{align}\label{eq:entropy}
\theta(t)\leq \gamma\|\rho_0\|_{L^1}t+\theta(0).
\end{align}
\end{enumerate}

$\mathrm{(iii)}$ For $\gamma>0$, $g(z,t)$ converges to the steady state:
\[
\lim_{t\to\infty}g(z,t)=-\sqrt{\gamma^2z^2-2\gamma},~~\forall z\in\mathbb{C}_+,
\]
and $\rho(x,t)$ converges to the steady state given by semicircle law:
\begin{align}\label{eq:steadypho}
\lim_{t\to\infty}\rho(x,t)=\rho_\infty(x):=\frac{\sqrt{(2\gamma -\gamma^2x^2)_+}}{\pi},~~\forall x\in\mathbb{R}.
\end{align}

$\mathrm{(iv)}$ For $\gamma=0$, 
the solution $g(z,t)$ and $\rho(x,t)$ converge to steady state after scaling in the following sense:
\[
e^{t} g\left(e^{t}z,\frac{e^{2t-1}}{2}\right)- z\to -\sqrt{z^2-2}~\textrm{ as }~t\to\infty. 
\]
and
\[
e^{t} \rho\left(e^{t} x, \frac{e^{2t}-1}{2}\right)\to \frac{\sqrt{(2-x^2)_+}}{\pi}~\textrm{ as }~t\to\infty. 
\]
\end{theorem}

We remark that  part (i) of Theorem \ref{thm:analytic} is derived directly by  combining the solutions given by \cite{Castro2008Global} and the space-time rescaling \eqref{eq:transform} as described below.
Consider the following complex Burgers equation
\begin{gather}\label{eq:complexBurgers3}
\left\{
\begin{split}
&[\partial_\tau \tilde{g}+ \tilde{g}\partial_w \tilde{g}](w,\tau)=0,~~~~w \in \mathbb{C}_+,\\
& \tilde{g}(w,0)=g_0(w)+\gamma w,
\end{split}
\right.
\end{gather}
where $g_0$ is defined by \eqref{eq:initialg}.
Castro and C\'ordoba \cite{Castro2008Global}  proved global existence and uniqueness of $\mathbb{C}_+$-holomorphic solution $\tilde{g}$  to \eqref{eq:complexBurgers3}   by the method of characteristics. For $t>0$, $\tilde{g}(\cdot,t)$ is $\occ$-holomorphic. Hence, from \eqref{eq:transform} we obtain  a $\mathbb{C}_+$-holomorphic solution $g$  to \eqref{eq:complexBurgers2} with initial datum $g_0$ and for $t>0$, $g(\cdot,t)$ is $\occ$-holomorphic. This proves part (i) of Theorem \ref{thm:analytic}. For part (ii), let
\[
f(z,t):=g(z,t)+\gamma z,~~z\in\occ,~~t>0.
\]
Then, $f$ is a to $\mathbb{C}_+$-holomorphic solution to \eqref{eq:complexBurgers1} with initial datum $f_0$ given  by \eqref{eq:initialf} and for $t>0$, $g(\cdot,t)$ is $\occ$-holomorphic. Consider the trace of $f$ on the real line and define
\[
f(x,t):=u(x,t)-i\pi\rho(x,t),~~x\in\mathbb{R},~~t>0.
\]
Then, we have $u=\pi H\rho$ and $\rho(x,t)$ is an analytical solution to the Dyson equation \eqref{eq:meanfield} with initial datum $\rho_0$. This proves part (ii) of Theorem \ref{thm:analytic}.

Since, the acceleration of characteristics for complex Burgers \eqref{eq:complexBurgers2} is not zero, which is different with \eqref{eq:complexBurgers3}. This also brings some detailed information of solutions. Therefore, for completeness and to unveil those information, we provide a direct proof for Theorem \ref{thm:analytic} in Appendix \ref{App_B}. 

\begin{remark}[Finite time blow up]\label{eq:blowup}
Note that condition $\rho_0>0$ is essential to Theorem \ref{thm:analytic}.
Castro and C\'ordoba \cite[Theorem 4.4, Remark 4.5]{Castro2008Global} proved that if $\rho_0\geq0$  and $\rho_0\in H^2(\mathbb{R})$, then there exists a unique local solution $\rho\in C([0,T];H^2(\mathbb{R}))\cap C^1([0,T];H^1(\mathbb{R}))$ to \eqref{eq:meanfield} with $\gamma=0$.
Moreover, if $\rho_0(x_0)=0=\inf_{x\in\mathbb{R}}\rho_0(x)$ for some point $x_0\in\mathbb{R}$, the solution blows up in finite time (see \cite[Theorem 4.8, Remark 4.9]{Castro2008Global}). Precisely,  along the trajectories of characteristics $X(x_0,t)$ starting from $x_0$, we have 
\[
X(x_0,t)=H\rho_0(x_0)t+x_0,
\]
and
\[
\partial_xH\rho(X(x_0,t),t)\to-\infty~\textrm{ as }~t\to t^*:=-\frac{1}{\partial_xH\rho_0(x_0)}.
\]
Due to  \eqref{eq:transform}, there exists a unique  local solution $\tilde{\rho}$ to \eqref{eq:meanfield} for $\gamma>0$ given by
\[
\tilde{\rho}(y,\tau)=e^{\gamma \tau} \rho\left(e^{\gamma\tau} y, \frac{e^{2\gamma \tau}-1}{2\gamma}\right),~~y\in\mathbb{R},~~\tau>0.
\]
Moreover, we have
\[
\partial_yH\tilde{\rho}(y,\tau)=e^{2\gamma\tau}\partial_xH\rho\left(e^{\gamma\tau} y, \frac{e^{2\gamma \tau}-1}{2\gamma}\right).
\]
Let 
\[
t=\frac{e^{2\gamma\tau}-1}{2\gamma},\quad y=e^{-\gamma\tau}X(x_0,t)=e^{-\gamma\tau}\left[\frac{H\rho_0(x_0)(e^{2\gamma\tau}-1)}{2\gamma}+x_0\right],
\]
and
\[
\tau^*=\frac{1}{2\gamma}\log(1+2\gamma t^*).
\]
Then, we have
\begin{align*}
\lim_{\tau\to\tau^*}\partial_yH\tilde{\rho}(y,\tau)&=\lim_{\tau\to\tau^*}e^{2\gamma \tau} \partial_xH\rho\left(e^{\gamma\tau} y, \frac{e^{2\gamma \tau}-1}{2\gamma}\right)\\
&=e^{2\gamma \tau^*}\lim_{t\to t^*} \partial_xH\rho\left(X(x_0,t), t\right)=-\infty.
\end{align*}
Hence, the solution to \eqref{eq:meanfield} with $\gamma>0$ also blows up in finite time.
\end{remark}

\subsection{Explicit solutions to the Dyson equation \eqref{eq:meanfield} from semicircle law and exponential convergence to the steady state for $\gamma>0$}\label{sec:explicit}
In this subsection, we give some explicit solutions to the Dyson equation \eqref{eq:meanfield} by using Wigner's semicircle law \eqref{eq:semicirclelaw}. When $\gamma>0$ the explicit solutions  converge exponentially to steady state given by \eqref{eq:steadypho}.

\subsubsection{An explicit solution to the Dyson equation  \eqref{eq:meanfield} with $\gamma=0$}
For $\gamma=0$, notice that $\sqrt{N}A(t)/\sqrt{t}$ is a Wigner matrix (Hermitian matrix with i.i.d entries which have mean zero and variance one), where $A(t)$ is defined by \eqref{eq:OUP} with $A(0)=0$. Let $\{\lambda_j(t)\}_{j=1}^N$ be the eigenvalues of matrix $A(t)$. Hence, as $N$ goes to infinity, the empirical measure $\frac{1}{N}\sum_{j=1}^N\delta_{\lambda_j(t)/\sqrt{t}}$ almost surely converges to Wigner's semicircle law $\mu_1(x)$ given by \eqref{eq:semicirclelaw} weakly in probability measure space (see \cite{wigner1955characteristic} or \cite[Theorem  2.4.2]{Tao2012Topics}). On the other hand,  the empirical measure $\mu^N(t)=\frac{1}{N}\sum_{j=1}^N\delta_{\lambda_j(t)}(x)$ almost surely converges to a measure solution $\rho(x,t)$ of the Dyson equation \eqref{eq:meanfield} with $\gamma=0$ \cite{rogers1993interacting}. We can obtain the relation between $\rho(x,t)$  and $\mu_1(x)$ by the following lemma.
\begin{lemma}\label{lmm:semicirclesolution}
For any constant $a>0$, if we have the following narrow convergences in probability measure space $\mathcal{P}(\mathbb{R})$:
\[
\tilde{\nu}^N(x):=\frac{1}{N}\sum_{j=1}^N\delta_{x_j/a}(x)\to \tilde{\nu}(x)~\textrm{ and }~\nu^N(x):= \frac{1}{N}\sum_{j=1}^N\delta_{x_j}(x)\to \nu(x)
\]
for two probability measures $\tilde{\nu},\nu$, then we have
\begin{align}\label{eq:relationmunu}
\nu(x)=\frac{1}{a}\tilde{\nu}\left(\frac{x}{a}\right).
\end{align}
\end{lemma}
\begin{proof}
For any test function $\varphi\in C_b(\mathbb{R})$, we have
\begin{align*}
\int_{\mathbb{R}}\varphi(x)\di \tilde{\nu}(x)&=\lim_{N\to\infty}\int_{\mathbb{R}}\varphi(x)\di \tilde{\nu}^N(x)=\frac{1}{N}\sum_{j=1}^N\varphi(x_j/a)\\
&=\lim_{N\to\infty}\int_{\mathbb{R}}\varphi(y/a)\di \nu^N(y)=a\lim_{N\to\infty}\int_{\mathbb{R}}\varphi(x)\di \nu^N(ax)\\
&=\int_{\mathbb{R}}\varphi(y/a)\di \nu(y)=a\int_{\mathbb{R}}\varphi(x)\di \nu(a x).
\end{align*}
Hence, $a\nu(ax)=\tilde{\nu}(x)$, which implies \eqref{eq:relationmunu}.

\end{proof}
From Lemma \ref{lmm:semicirclesolution}, we choose $\rho$ as the rescaling of $\rho_1$ defined in \eqref{eq:semicirclelaw}
\begin{align}\label{eq:explicitsolution}
\rho(x,t)=\frac{1}{\sqrt{t}}\rho_1\left(\frac{x}{\sqrt{t}}\right)=\frac{\sqrt{(4t-x^2)_+}}{2\pi t},
\end{align}
where $\rho(x,t)$ is the limit of the empirical measure $\frac{1}{N}\sum_{j=1}^N\delta_{\lambda_j(t)}(x)$ for $\gamma=0$. This implies  $\rho(x,t)$ is  a kind of self-similar rarefaction wave solution of the Dyson equation \eqref{eq:meanfield} with $\gamma=0$.
Next, we calculate $u(x,t)$ using the Hilbert transform of $\pi \rho(x,t)$ and then verify the obtained $(\rho, u)$ satisfies \eqref{eq:meanfield} ($\gamma=0$) . For $x\in\mathbb{R}\setminus[-2\sqrt{t},2\sqrt{t}]$, by changing of variable with $y=2\sqrt{t}\sin\theta$,  we have
\begin{align}\label{eq:calus}
(\pi H\rho)(x,t)=&\frac{1}{2t\pi}\int_{-2\sqrt{t}}^{2\sqrt{t}}\frac{\sqrt{4t-y^2}}{x-y}\di y\nonumber\\
=&\frac{1}{2t\pi}\int_{-\pi/2}^{\pi/2}\left(x+2\sqrt{t}\sin\theta\right)\di \theta +\frac{4t-x^2}{2t\pi}\int_{-\pi/2}^{\pi/2}\frac{1}{x-2\sqrt{t}\sin\theta}\di \theta\nonumber\\
=&\frac{x}{2t} +\frac{4t-x^2}{2t\pi}\int_{-\pi/2}^{\pi/2}\frac{1}{x-2\sqrt{t}\sin\theta}\di \theta\nonumber\\
=&\frac{x}{2t} +\frac{4t-x^2}{2t\pi}\frac{2}{\sqrt{x^2-4t}}\left[\arctan\left(\frac{x-2\sqrt{t}}{\sqrt{x^2-4t}}\right)+\arctan\left(\frac{x+2\sqrt{t}}{\sqrt{x^2-4t}}\right)\right].
\end{align}
Using the fact
\begin{gather*}
\arctan x+\arctan y =\left\{
\begin{split}
\frac{\pi}{2}~\textrm{ for }~x\cdot y=1,~~x,y>0,\\
-\frac{\pi}{2}~\textrm{ for }~x\cdot y=1,~~x,y<0,
\end{split}
\right.
\end{gather*}
we obtain
\begin{gather}\label{eq:Htran2}
(\pi H\rho)(x,t)=\left\{
\begin{split}
&\frac{x}{2t} +\frac{\sqrt{x^2-4t}}{2t},~~x<-2\sqrt{t},\\
&\frac{x}{2t} -\frac{\sqrt{x^2-4t}}{2t},~~x>2\sqrt{t}.
\end{split}
\right.
\end{gather}
For $x\in[-2\sqrt{t},2\sqrt{t}]$, we have 
\begin{align*}
(\pi H\rho)(x,t)=&\frac{1}{2t\pi}\p\int_{-2\sqrt{t}}^{2\sqrt{t}}\frac{\sqrt{4-y^2}}{x-y}\di y\\
=&\frac{1}{2t\pi}\lim_{\e\to0}\left(\int_{-2\sqrt{t}}^{x-\e}+(\int^{2\sqrt{t}}_{x+\e}\right)\frac{\sqrt{4-y^2}}{x-y}\di y.
\end{align*}
Then, using similar calculation as \eqref{eq:calus} we have
$
(\pi H\rho)(x,t)=\frac{x}{2t},~~x\in[-2\sqrt{t},2\sqrt{t}].
$
Therefore we have
\begin{gather}\label{app7}
u(x,t)=(\pi H\rho)(x,t)=\left\{
\begin{split}
&\frac{x+\sqrt{x^2-4t}}{2t},~~x<-2\sqrt{t},\\
&\frac{x}{2t},~~x\in[-2\sqrt{t},2\sqrt{t}],\\
&\frac{x-\sqrt{x^2-4t}}{2t},~~x>2\sqrt{t},
\end{split}
\right.
\end{gather}
and  $(\rho,u)$ satisfies \eqref{eq:meanfield} ($\gamma=0$) with initial datum
\begin{equation}\label{eq:initialdatumexplicit}
\rho(x,0)= \delta(0),~~u(x,0)=(\pi H\rho)(x,0)=\mathrm{p.v.}\frac{1}{x}.
\end{equation}
Notice that the above self-similar solution $(\rho,u)$ corresponds to the self-similar solution to complex Burgers equation given in \cite[Section 1.2]{Menon1}.

In Appendix \ref{app_A} we will give the same  explicit solution by the Stieltjes transform of Wigner's semicircle law $\mu_1$  (see  \eqref{eq:Lowurho} ).

\begin{remark}[Connection with Barenblatt solutions to porous media equation]\label{rmk:Barenblatt}
Consider the following one dimensional porous media equation:
\[
\partial_th=\frac{\pi^2}{3}\partial_{xx}(h^3),\quad h|_{t=0}=\delta(0).
\]
It has a self-similar solution called Barenblatt solution (see \cite[Page 104]{salsa2016partial}) given by
\[
h(x,t)=\frac{\sqrt{(4\sqrt{t}-x^2)_+}}{2\pi \sqrt{t}}=\frac{1}{t^{1/4}}\cdot \frac{\sqrt{\left(4-\left(\frac{x}{t^{1/4}}\right)^2\right)_+}}{2\pi }.
\]
Notice that
\[
\rho(x,t)=h(x,t^2)=\frac{\sqrt{(4t-x^2)_+}}{2\pi t}
\]
is exactly the explicit solution \eqref{eq:explicitsolution} to the Dyson equation \eqref{eq:meanfield} with $\gamma=0$.
\end{remark}

\subsubsection{An explicit solution to the Dyson equation   \eqref{eq:meanfield} with $\gamma>0$ and exponential convergence to the steady state}
When $\gamma>0$, we first show that \eqref{eq:explicitsolution} with $t=\frac{1}{2\gamma}$ gives a steady state of \eqref{eq:meanfield} with $\gamma>0$. Actually, we have
\[
\rho\left(x,\frac{1}{2\gamma}\right)=\frac{\sqrt{(2\gamma-\gamma^2 x^2)_+}}{\pi},
\]
and
\begin{gather*}
u\left(x,\frac{1}{2\gamma}\right)=\pi H\rho\left(x,\frac{1}{2\gamma}\right)=\left\{
\begin{split}
&\gamma x+\sqrt{\gamma^2x^2-2\gamma},~~x<-\sqrt{2},\\
& \gamma x,\quad x\in[-\sqrt{2},\sqrt{2}],\\
&\gamma x-\sqrt{\gamma^2x^2-2\gamma},~~x>\sqrt{2}.
\end{split}
\right.
\end{gather*}
Define
\begin{align}\label{eq:steadystate}
\rho_\infty(x):=\rho\left(x,\frac{1}{2\gamma}\right),~~u_\infty(x)=u\left(x,\frac{1}{2\gamma}\right),
\end{align}
and then
\[
\rho_\infty(u_\infty- \gamma x)\equiv0,
\]
which implies that $\rho_\infty$ is a steady state of the Dyson equation \eqref{eq:meanfield} when $\gamma>0$. Due to the convexity of the energy $E$ in \eqref{eq:interactionEnergy}, the steady state is the minimizer and it is unique (see Remark \ref{rmk:convergence}).

Next, we construct an explicit solution which converges to $\rho_\infty$ exponentially. Let $\sigma(t)$ be an unknown function and $\sigma(0)=\sigma_0>0$ and assume  solution $\rho(x,t)$ to \eqref{eq:meanfield} with $\gamma>0$ has the following form
\begin{align}\label{eq:exponentialrhot}
\rho(x,t)=\frac{\sqrt{(2\sigma(t)-x^2)_+}}{\pi \sigma(t)}.
\end{align}
Correspondingly, we have
\begin{gather*}
u(x,t)=\pi H\rho(x,t)=\left\{
\begin{split}
&\frac{x+\sqrt{x^2-2\sigma(t)}}{\sigma(t)},\quad x<-\sqrt{2\sigma(t)},\\
&\frac{x}{\sigma(t)},\quad\quad \quad  |x|\leq \sqrt{2\sigma(t)},\\
&\frac{x-\sqrt{x^2-2\sigma(t)}}{\sigma(t)},\quad x>\sqrt{2\sigma(t)}.
\end{split}
\right.
\end{gather*}
Obviously, $(\rho, u)$ satisfies \eqref{eq:meanfield} when $|x|>\sqrt{2\sigma(t)}$. Next, we consider the case $|x|\leq \sqrt{2\sigma(t)}$ to obtain a proper ordinary differential equation for $\sigma(t)$ such that $(\rho,u)$ is a solution of \eqref{eq:meanfield}. Direct calculations show that
\[
\partial_t\rho=-\frac{\sqrt{2\sigma-x^2}}{\pi\sigma^2}\dot{\sigma}+\frac{\dot{\sigma}}{\pi\sigma\sqrt{2\sigma-x^2}},\quad \partial_x\rho=-\frac{x}{\pi\sigma\sqrt{2\sigma-x^2}},
\]
and
\[
\rho+x\partial_x\rho=2\sigma\left(\frac{\sqrt{2\sigma-x^2}}{\pi \sigma^2}-\frac{1}{\pi\sigma\sqrt{2\sigma-x^2}}\right).
\]
Take the above equalities into  \eqref{eq:meanfield} and we obtain
\begin{align*}
&\partial_t\rho+\partial_x[\rho(u-\gamma x)]=\partial_t\rho+\left(\frac{1}{\sigma}-\gamma\right)(\rho+x\partial_x\rho)\nonumber\\
=&(-\dot{\sigma}+2-2\gamma\sigma)\left(\frac{\sqrt{2\sigma-x^2}}{\pi \sigma^2}-\frac{1}{\pi\sigma\sqrt{2\sigma-x^2}}\right)=0,~~|x|\leq \sqrt{2\sigma(t)}.
\end{align*}
Hence, we have
\[
\dot{\sigma}(t)=2-2\gamma\sigma,~~\sigma(0)=\sigma_0>0,
\]
which implies
\[
\sigma(t)=\frac{1}{\gamma}-\frac{1-\gamma\sigma_0}{\gamma}e^{-2 \gamma t}>0.
\]
Hence, for any $\sigma_0>0$, an explicit solution to \eqref{eq:meanfield}  is given by
\begin{align}\label{eq:explicitrho}
\rho(x,t)=\frac{\sqrt{\left(2\gamma [1-(1-\gamma\sigma_0)e^{-2\gamma t}]-\gamma^2x^2\right)_+}}{\pi [1-(1-\gamma \sigma_0)e^{-2\gamma t}]}.
\end{align}
This solution tends to $\rho_\infty$ (defined by \eqref{eq:steadystate}) exponentially as $t\to\infty$.

\subsection{Global weak solutions of the Dyson equation \eqref{eq:meanfield}}\label{sec:weaksolutions}
In Theorem \ref{thm:analytic}, we proved global existence and uniqueness of a positive analytical solution to \eqref{eq:meanfield} with a strictly positive initial datum $\rho_0>0$ and $\rho_0\in H^s(\mathbb{R})\cap L^1(\mathbb{R})$ with $s>1/2$. If $\rho_0\geq0$ and $\rho_0(x_0)=0$ for some $x_0\in\mathbb{R}$, 
the solution to \eqref{eq:meanfield} blows up in finite time (see Remark \ref{eq:blowup}) in the sense that $\partial_xH\rho$ goes to $-\infty$. Consequently, there is also a finite time blow up in the space $H^s(\mathbb{R})$ for $s>3/2$.  Next, we show global existence of weak solution in $\dot{H}^{1/2}(\mathbb{R})\cap L^1(\mathbb{R})$. Note that we have interpolation inequality
 \[
 \|\rho\|_{L^2}\leq 3\|\rho\|_{L^1}^{1/2}\|\rho\|^{1/2}_{\dot{H}^{1/2}}.
 \]
 Hence $\rho \in \dot{H}^{1/2}(\mathbb{R})\cap L^1(\mathbb{R})$ is equivalent to $\rho \in H^{1/2}(\mathbb{R})\cap L^1(\mathbb{R})$. Let us define the weak solutions:
\begin{definition}\label{def:weak1}
For $T>0$, $\rho_0\in H^{1/2}(\mathbb{R})\cap L^1(\mathbb{R})$ and $\rho_0\geq0$, a nonnegative function $\rho\in L^\infty(0,T;H^{1/2}(\mathbb{R})\cap L^1(\mathbb{R}))\cap W^{1,\infty}(0,T;H^{-m}(\mathbb{R}))$ for some $m>0$  is said to be a weak solution of the Dyson equation \eqref{eq:meanfield} if
\begin{multline}\label{eq:defweak1}
\int_0^T\int_{\mathbb{R}}\partial_t\phi(x,t)\rho( x,t) \,\di x\di t+\int_{\mathbb{R}}\phi(x,0)\rho_0(x)\di x\\
=-\frac{1}{2}\int_0^T\int_{\mathbb{R}}\int_{\mathbb{R}}\frac{\partial_x\phi(x,t)-\partial_x\phi(y,t)}{x-y}\rho(x,t)\rho(y,t)\,\di x\di y\di t\\
+\gamma\int_0^T\int_{\mathbb{R}}x\partial_x\phi(x,t)\rho(x,t)\di x\di t,
\end{multline}
holds for any test function $\phi\in C_c^\infty(\mathbb{R}\times[0,T))$.
\end{definition}

\begin{theorem}\label{thm:weaktheorem}
Assume $0\leq \rho_0\in H^{1/2}(\mathbb{R})\cap L^1(\mathbb{R})$ and $m_2(0):=\int_{\mathbb{R}}x^2\rho_0(x)\di x<\infty$. Then, there exists a unique global nonnegative weak solution to the Dyson equation \eqref{eq:meanfield} satisfying
\[
\rho\in L^\infty(0,T; H^{1/2}(\mathbb{R})\cap L^1(\mathbb{R}))\cap W^{1,\infty}(0,T;H^{-3}(\mathbb{R}))
\]
for any time $T>0$. Moreover, we have the following estimates
\begin{enumerate}
\item[(a)] 
\begin{align}\label{eq:solutionestimate}
\|\rho(t)\|_{H^{1/2}}\leq e^{\gamma t}\|\rho_0\|_{H^{1/2}},~~t>0,
\end{align}

\item[(b)] The mass $\|\rho(t)\|_{L^1}$ is conserved:
\begin{align}\label{eq:L1conserved1}
\|\rho(t)\|_{L^1}=\|\rho_0\|_{L^1}.
\end{align}
\item[(c)] For a.e. $t>0$, the second moment $m_2(t):=\int_{\mathbb{R}}x^2\rho(x,t)\di x$ satisfies
\begin{gather}\label{eq:secondmoment1}
m_2(t)\leq \left\{
\begin{split}
&\frac{\|\rho_0\|_{L^1}^2}{2\gamma}-\frac{\|\rho_0\|_{L^1}^2-2\gamma m_2(0)}{2\gamma}e^{-2\gamma t},~~\gamma>0,\\
&m_2(0)+\|\rho_0\|_{L^1}^2t,~~\gamma=0.
\end{split}
\right.
\end{gather}
\item[(d)] The following energy dissipation holds:
\begin{align}\label{eq:energydissipation2}
E(\rho(\cdot,t))+\int_0^t\int_{\mathbb{R}}\rho(x,s)\big|\gamma x-\pi H\rho(x,s)\big|^2 \di x\di s\leq E(\rho_0)~\textrm{ for any }~t>0,
\end{align}
with $E$ defined by \eqref{eq:interactionEnergy}. 
\item[(e)] If $\rho_0\log\rho_0\in L^1(\mathbb{R})$, then the entropy $\theta(t):=\int_{\mathbb{R}}\rho(x,t)\log\rho(x,t)\di x$ satisfies
\begin{align}\label{eq:entropy2}
\theta(t)\leq \gamma\|\rho_0\|_{L^1}t+\theta(0),\quad t>0.
\end{align}
\end{enumerate}

\end{theorem}
\begin{proof}
Let $\varphi_\e>0$ ($\e>0$) be the standard Friedrichs mollifier. Set
\[
\rho_{0}^\e=\rho_0\ast\varphi_\e.
\]
Then, for nontrival initial datum $\rho_0$, we have $\rho_{0}^\e(x)>0$ for $x\in\mathbb{R}$ and $\rho_{0}^\e\in H^s(\mathbb{R})\cap L^1(\mathbb{R})$ ($s>1/2$). Moreover, from Young's inequality for convolution, we have
\begin{align}\label{eq:rhoe}
\|\rho_0^\e\|_{L^2}\leq \|\rho_0\|_{L^2},\quad \|\rho_0^\e\|_{\dot{H}^{1/2}}\leq \|\rho_0\|_{\dot{H}^{1/2}},\quad \|\rho_0^\e\|_{L^1}=\|\rho_0\|_{L^1}.
\end{align}
By Theorem \ref{thm:analytic}, we have a global positive analytical solution $\rho^\e$ to \eqref{eq:meanfield} with initial date $\rho_0^\e$:
\begin{align}\label{eq:meanfield1}
\partial_t\rho^\e+\partial_x[\rho^\e (\pi H\rho^\e-\gamma x)]=0.
\end{align}

\textbf{Step 1.} Uniform estimates for $\rho^\e$.

First, multiplying \eqref{eq:meanfield1} by $\rho^\e$ and integration by parts show that
\begin{align*}
\frac{\di }{\di t} \int_{\mathbb{R}} \frac{(\rho^\e)^2}{2} \di x + \frac{\pi}{2}\int_{\mathbb{R}} (\rho^\e)^2 \partial_xH\rho^\e\di x-\frac{\gamma}{2}\int_{\mathbb{R}}(\rho^\e)^2\di x=0.
\end{align*}
Since the second term on the left hand side is
\begin{align*}
\frac{\pi}{2}\int_{\mathbb{R}} (\rho^\e)^2 \partial_xH\rho^\e\di x&=\int_{\mathbb{R}}\int_{\mathbb{R}} (\rho^\e(x,t))^2 \frac{\rho^\e(x,t)-\rho^\e(y,t)}{|x-y|^2} \di y \di x\\
&= \int_{\mathbb{R}} \int_{\mathbb{R}} \rho^\e(x) \frac{|\rho^\e(x,t)-\rho^\e(y,t)|^2}{|x-y|^2} \di x\di y,   
\end{align*}
we obtain
\begin{equation}\label{eq:L2}
\frac{\di }{\di t} \int_{\mathbb{R}}{(\rho^\e)^2} \di x + 2\int_{\mathbb{R}} \int_{\mathbb{R}} \rho^\e(x,t) \frac{|\rho^\e(x,t)-\rho^\e(y,t)|^2}{|x-y|^2} \di x\di y = \gamma\int_{\mathbb{R}}(\rho^\e)^2\di x.
\end{equation}
Gr\"onwall's inequality and \eqref{eq:rhoe} imply
\begin{align}\label{eq:L2estimate}
\|\rho^\e(t)\|_{L^2}^2\leq e^{\gamma t}\|\rho_0^\e\|_{L^2}^2\leq e^{\gamma t}\|\rho_0\|_{L^2}^2,~~t>0.
\end{align}
Second, multiplying \eqref{eq:meanfield1} by $H\rho_x^\e$ gives the following estimate:
\begin{align}\label{eq:H120}
\frac{1}{2} \frac{\di}{\di t} \|(-\Delta)^{1/4} \rho^\e\|_{L^2}^2 + \pi\int_{\mathbb{R}^2}(\partial_xH\rho^\e)^2 \rho^\e \di x+ \pi\int_{\mathbb{R}} \partial_x\rho^\e \partial_xH \rho^\e H \rho^\e\di x-\gamma\int_{\mathbb{R}}\partial_xH\rho^\e\partial_x(x\rho^\e)=0.
\end{align}
On the one hand, we have
\begin{align}\label{eq:term3}
\pi\int_{\mathbb{R}} \partial_x\rho^\e \partial_xH \rho^\e H \rho^\e\di x=- \pi\int_{\mathbb{R}} H( \partial_x\rho^\e \partial_xH \rho^\e) \rho^\e\di x =-\frac{\pi}{2} \int_{\mathbb{R}}  [(\partial_xH\rho^\e)^2 - (\partial_x\rho^\e)^2]\rho^\e\di x.
\end{align}
On the other hand, we estimate the last term in \eqref{eq:H120} as below.
Due to \eqref{eq:formula}, we derive 
\begin{align*}
\gamma\int_{\mathbb{R}}\partial_xH\rho^\e\partial_x(x\rho^\e)=&-\gamma\int_{\mathbb{R}}\partial_x\rho^\e\partial_xH(x\rho^\e)=-\gamma\int_{\mathbb{R}}\partial_x\rho^\e\partial_x(xH\rho^\e)\di x\\
&=-\gamma\int_{\mathbb{R}}\partial_x\rho^\e H\rho^\e\di x-\gamma\int_{\mathbb{R}}x\partial_x\rho^\e\partial_xH\rho^\e\di x\\
&=\gamma\|(-\Delta)^{1/4}\rho^\e\|_{L^2}^2-\gamma\int_{\mathbb{R}}x\partial_x\rho^\e\partial_xH\rho^\e\di x.
\end{align*}
Use \eqref{eq:formula} again and we have
\[
-\gamma\int_{\mathbb{R}}x\partial_x\rho^\e\partial_xH\rho^\e\di x=\gamma\int_{\mathbb{R}}H(x\partial_x\rho^\e) \partial_x\rho^\e\di x=\gamma\int_{\mathbb{R}}x\partial_xH\rho^\e\partial_x\rho^\e\di x.
\]
This implies $\gamma\int_{\mathbb{R}}x\partial_x\rho^\e\partial_xH\rho^\e\di x=0$ and hence
\begin{align}\label{eq:gammaterm}
\gamma\int_{\mathbb{R}}\partial_xH\rho^\e\partial_x(x\rho^\e)=\gamma\|(-\Delta)^{1/4}\rho^\e\|_{L^2}^2.
\end{align}
Combining \eqref{eq:H120}, \eqref{eq:term3} and \eqref{eq:gammaterm} shows
\begin{equation}\label{eq:H12}
\frac{\di}{\di t} \|(-\Delta)^{1/4} \rho^\e\|_{L^2}^2 + \pi\int_{\mathbb{R}}(\partial_xH\rho^\e)^2 \rho^\e\di x + \pi\int_{\mathbb{R}} \rho^\e (\partial_x\rho^\e)^2\di x=2\gamma\|(-\Delta)^{1/4}\rho^\e\|_{L^2}^2.
\end{equation}
Gr\"onwall's inequality and \eqref{eq:rhoe} imply
\begin{align}\label{eq:Hdot}
\|\rho^\e(t)\|_{\dot{H}^{1/2}}^2\leq e^{2\gamma t}\|\rho_0\|_{\dot{H}^{1/2}}^2, \quad t>0.
\end{align}
Inequalities \eqref{eq:L2estimate} and \eqref{eq:Hdot} yield 
\begin{align}\label{lions1}
\|\rho^\e(t)\|_{{H}^{1/2}}^2\leq e^{2\gamma t}\|\rho_0\|_{{H}^{1/2}}^2,  \quad t>0.
\end{align}
and hence we have
\[
\rho^\e\in L^\infty(0,T;H^{1/2}(\mathbb{R}))~\textrm{ for any }~T>0.
\]

Third, for time regularity, the following estimate holds for any $\phi\in C_c^\infty(\mathbb{R})$
\begin{align*}
&\int_{\mathbb{R}}\phi(x)\partial_t\rho^\e(x,t)\di x \\
=&-\frac{1}{2} \int_{\mathbb{R}}\int_{\mathbb{R}}\frac{\partial_x\phi(x)-\partial_x\phi(y)}{x-y}\rho^\e(x,t){\rho}^\e(y,t)\di x\di y
+\gamma\int_{\mathbb{R}}x\phi_x(x,t){\rho^\e}( x,t)\di x\\
\leq& C(\|\partial_{xx}\phi\|_{L^\infty}+(m_2(\rho^\e) + \|\rho^\e\|_{L^1})\|\partial_{x}\phi\|_{L^\infty})\leq C\|\phi\|_{H^3},
\end{align*}
and hence
\begin{align}\label{eq:time}
\|\partial_t{\rho}^\e\|_{L^\infty(0,\infty;H^{-3}(\mathbb{R}))}\leq C, \quad 
\partial_t{\rho}^\e\in L^\infty(0,\infty;H^{-3}(\mathbb{R})).
\end{align}

\textbf{Step 2.} Take limits for $\rho^\e$ as $\e$ goes to $0$.

First, from uniform estimates \eqref{lions1} and \eqref{eq:time} in Step 1,  there exist $\rho\in L^\infty(0,T; H^{1/2}(\mathbb{R}))\cap W^{1,\infty}(0,T;H^{-3}(\mathbb{R}))$ and a subsequence of $\{{\rho}^\e\}_{\e>0}$ (still denoted as $\{{\rho}^\e\}_{\e>0}$) such that 
\[
\rho^\e\overset{\ast}{\rightharpoonup} \rho~\textrm{ in }~L^\infty(0,T;H^{1/2}(\mathbb{R}))~\textrm{ as }~\e\to0,
\]
and
\[
\partial_t\rho^\e\overset{\ast}{\rightharpoonup} \partial_t\rho~\textrm{ in }~L^\infty(0,T;H^{-3}(\mathbb{R}))~\textrm{ as }~\e\to0.
\]
Hence, we have \eqref{eq:solutionestimate}.

Second, from \eqref{lions1}  and \eqref{eq:time}, by Lions-Aubin Lemma, we also know
\[
{\rho}^\e\to \rho~\textrm{ in }~L^\infty(0,T;L_{loc}^2(\mathbb{R}))~\textrm{ as }~\e\to0,
\]
and consequently
\begin{equation}\label{strong2}
{\rho}^\e\to \rho~\textrm{ in }~L^\infty(0,T;L_{loc}^1(\mathbb{R}))~\textrm{ as }~\e\to0.
\end{equation}
Due to $\|\rho^\e(t)\|_{L^1}\equiv\|\rho_0\|_{L^1}$, we have 
\[
\|\rho(t)\|_{L^1}=\lim_{R\to+\infty}\int_{-R}^R\rho(x,t)\di x=\lim_{R\to+\infty}\lim_{\e\to0}\int_{-R}^R\rho^\e(x,t)\di x\equiv\|\rho_0\|_{L^1},
\]
where the last step we used the uniform bound of second momentum for $\rho^\e$ \eqref{eq:secondmoment}.
Hence, 
\[
\rho\in L^\infty(0,T;L^1(\mathbb{R})),
\]
and \eqref{eq:L1conserved1} holds. 
For any test function $\phi\in C_c^\infty(\mathbb{R}\times [0,T))$, by \eqref{eq:rhoe} we have
\begin{multline}\label{eq:defweak}
\int_0^T\int_{\mathbb{R}}\partial_t\phi(x,t){\rho}^\e( x,t) \,\di x\di t+\int_{\mathbb{R}}\phi(x,0)\rho_0^\e(x)\di x\\
=-\frac{1}{2}\int_0^T\int_{\mathbb{R}}\int_{\mathbb{R}}\frac{\partial_x\phi(x,t)-\partial_x\phi(y,t)}{x-y}{\rho}^\e(x,t){\rho}^\e(y,t)\,\di x\di y\di t\\
+\gamma\int_0^T\int_{\mathbb{R}}x\partial_x\phi(x,t){\rho}^\e(x,t)\di x\di t,
\end{multline}
By the strong convergence of ${\rho}^\e$ in \eqref{strong2}, we can take the limit as $\e\to0$ in \eqref{eq:defweak} and conclude that $\rho$ satisfies \eqref{eq:defweak1}. Hence, $\rho$ is a global weak solution to \eqref{eq:meanfield}.

\textbf{Step 3.} Consequent estimates for $H\rho$.
First, from \eqref{lions1} and
\begin{equation}
\|(-\Delta)^{\frac14} \rho^\e\|_{L^2}^2 = \int_{\mathbb{R}} (H\rho^\e) H\partial_x (H\rho^\e) \di x= \|(-\Delta)^{\frac14} (H\rho^\e)\|_{L^2}^2,
\end{equation}
we have uniform estimates
\begin{equation}\label{equ:high}
\|H \rho^\e\|_{L^{\infty}(0,T; H^{\frac12}(\mathbb{R}))} \leq C \text{ for any }T>0.
\end{equation}
Second, from the equation for $u^\e$ \eqref{eq:ubeta1} with $u^\e=\pi H \rho^\e$, we have
for any $\phi\in C_c^\infty(\mathbb{R})$
\begin{align*}
&\int_{\mathbb{R}}\phi(x)\partial_t u^\e(x,t)\di x 
=\int_{\mathbb{R}} \partial_x \phi(x) \left[\frac{(u^\e)^2}{2} - \frac{\pi^2}{2} (\rho^\e)^2-\gamma x u^\e\right] \di x\\
\leq& C\|\partial_{x}\phi\|_{L^\infty} \|\rho^\e\|_{L^2}^2+
\int_{\mathbb{R}} \gamma H(x \partial_x \phi) \rho^\e \di x
=  C\|\partial_{x}\phi\|_{L^\infty} \|\rho^\e\|_{L^2}^2+
\int_{\mathbb{R}} \gamma H( \partial_x \phi) x  \rho^\e \di x\\
\leq& C[\|\partial_{x}\phi\|_{L^\infty} \|\rho^\e\|_{L^2}^2+ (m_2(\rho^\e)^{\frac12} \|\rho^\e\|_{L^2}^{\frac12})\|H\partial_x \phi\|_{L^4}]\leq C(\|\partial_{x}\phi\|_{L^\infty}+ \|\partial_x \phi\|_{L^4} ),
\end{align*}
and hence
\begin{align}\label{equ:time}
\|\partial_t{u}^\e\|_{L^\infty(0,\infty;H^{-2}(\mathbb{R}))}\leq C, \quad 
\partial_t{u}^\e\in L^\infty(0,\infty;H^{-2}(\mathbb{R})).
\end{align}
Similar to $\rho^\e$, combining \eqref{equ:high}, \eqref{equ:time} and Lions-Aubin Lemma, we also know for $u=\pi H \rho\in L^\infty(0,T; H^{1/2}(\mathbb{R}))\cap W^{1,\infty}(0,T;H^{-2}(\mathbb{R}))$,
\begin{align}
u^\e\overset{\ast}{\rightharpoonup} u~\textrm{ in }~L^\infty(0,T;H^{1/2}(\mathbb{R}))~\textrm{ as }~\e\to0,\\
\partial_t u^\e\overset{\ast}{\rightharpoonup} \partial_t u~\textrm{ in }~L^\infty(0,T;H^{-2}(\mathbb{R}))~\textrm{ as }~\e\to0,\\
{u}^\e\to  u~\textrm{ in }~L^\infty(0,T;L_{loc}^2(\mathbb{R}))~\textrm{ as }~\e\to0.
\end{align}
Consequently, we have for a.e. $t\in[0,T]$
\begin{equation}\label{ae-cH}
H\rho^\e(\cdot,t)\to H \rho(\cdot,t)~\textrm{ for a.e. }x\in\mathbb{R}~\textrm{ as }~\e\to0.
\end{equation}

\textbf{Step 4.} The uniqueness of weak solutions is a direct result of the contraction property of Wasserstein distance as stated in \eqref{eq:contraction}.

\textbf{Step 5.} We prove properties \eqref{eq:secondmoment1}, \eqref{eq:energydissipation2}, and \eqref{eq:entropy2} below. 

Due to \eqref{eq:secondmoment}, we have 
\begin{gather}\label{mom2}
m^\e_2(t)=\left\{
\begin{split}
&\frac{\|\rho_0\|_{L^1}^2}{2\gamma}-\frac{\|\rho_0\|_{L^1}^2-2\gamma m^\e_2(0)}{2\gamma}e^{-2\gamma t},~~\gamma>0,\\
&m^\e_2(0)+\|\rho_0\|_{L^1}^2t,~~\gamma=0,
\end{split}
\right.
\end{gather}
where
\[
m^\e_2(t):=\int_{\mathbb{R}}x^2\rho^\e(x,t)\di x.
\]
Due to strong convergence of $\rho^\e$ to $\rho$ in $L^\infty(0,T;L^1_{loc}(\mathbb{R}))$, for a.e. $t\in(0,T)$ we have 
\begin{equation}\label{ae-c}
\rho^\e(\cdot,t)\to\rho(\cdot,t)~\textrm{ for a.e. }x\in\mathbb{R}~\textrm{ as }~\e\to0.
\end{equation}
To take the limit in \eqref{mom2}, first notice $m_2^\e(0)\to m_2(0)$ by Young's convolution inequality. Second, by Levi's lemma and Fatou's lemma,  we have 
\begin{equation}
\begin{aligned}
m_2(t) = & \lim_{N\to +\infty} \int_{\mathbb{R}}(x^2)_N \rho  \di x
\leq   \lim_{N\to +\infty}    \liminf_{\e\to 0} \int_{\mathbb{R}}(x^2)_N \rho^\e \di x\\\leq& \lim_{N\to +\infty}    \liminf_{\e\to 0} \int_{\mathbb{R}} x^2 \rho^\e \di x \leq \liminf_{\e\to 0} m_2^\e(t),
\end{aligned}
\end{equation}
where $(x^2)_N$ means the cutoff $(x^2)_N= \min\{x^2, N\}$.
Hence, we obtain \eqref{eq:secondmoment1}. 

For the energy dissipation \eqref{eq:energydissipation2}, we prove it by taking limit in  \eqref{eq:energydissipation1}, Levi's lemma and Fatou's Lemma. First  by pointwise convergence of $\rho^\e$ in \eqref{ae-c}, pointwise convergence of $H\rho^\e$ in \eqref{ae-cH}  and Fatou's lemma, we have
\begin{equation}
\int_0^T \int_{\mathbb{R}} \rho |\gamma x- \pi H \rho(x,t)|^2 \di x \leq \liminf_{\e\to 0}  \int_0^T \int_{\mathbb{R}} \rho^\e |\gamma x- \pi H \rho^\e(x,t)|^2 \di x .
\end{equation}
Second, there exists a constant $c$ such that 
$K(x,y):= \frac12\gamma (x^2+ y^2) + \log \frac{1}{|x-y|}+ c\geq 0$, so we
rewire the energy as 
\begin{equation}
E(\rho)= \frac12 \int_{\mathbb{R}^2 } K(x,y) \rho(x) \rho(y) \di x \di y-\frac{c}{2}.
\end{equation}
Denote the cutoff of $K$ as $K_N(x,y):= \min\{K(x,y), N\}$ such that $0\leq K_N(x,y)\leq  K(x,y)$, which increasingly converges to $K(x,y)$ for a.e. $(x,y)\in \mathbb{R}^2$.  Then by  Levi's lemma and Fatou's Lemma, we obtain
\begin{equation}
\begin{aligned}
E(\rho) + \frac{c}{2} = & \lim_{N\to + \infty} \frac12\int_{\mathbb{R}^2} K_N(x,y) \rho(x) \rho(y) \di x \di y\\
\leq & \lim_{N \to +\infty} \liminf_{\e \to 0} \frac12 \int_{\mathbb{R}^2} K_N(x,y) \rho^\e(x) \rho^\e(y) \di x \di y\\
\leq & \lim_{N \to +\infty} \liminf_{\e \to 0} \frac12 \int_{\mathbb{R}^2} K(x,y) \rho^\e(x) \rho^\e(y) \di x \di y\\
\leq & \liminf_{\e \to 0}  E(\rho^\e) + \frac{c}{2}.
\end{aligned}
\end{equation}

The entropy inequality \eqref{eq:entropy2} can be obtained by \eqref{eq:entropy} and the weak lower semi-continuity of the entropy \cite{jordan1998variational}.

\end{proof}

\begin{remark}
We shall remark that the global existence of weak solutions for the following nonconservative equation remains open:
\[
\partial_t\rho-u\partial_x\rho=0,~~u=H\rho.
\]
We refer to \cite{cordoba2005formation,Silvestre} for in depth study of this equation with or without a viscous term.
\end{remark}

\begin{remark}[Exponential convergence to the steady state]\label{rmk:convergence}
Carrillo et. al. \cite{carrillo2012mass} proved the existence and uniqueness of probability solutions by using gradient flow structure in Wasserstein distance.
Notice  the free energy $E(\rho)$ given by \eqref{eq:interactionEnergy} for the Dyson equation consists a harmonic trap  energy $E_{\textnormal{h}}$ and an interaction energy
$E_{\textnormal{i}}$. $E_{\textnormal{i}}$ is convex (or displacement convex) along generalized Wasserstein geodesics and $E_{\textnormal{h}}$ is $\gamma$-convex along Wasserstein geodesics as explained below.
Assume $\rho_0,\rho_1\in\mathcal{P}_{\text{AC}}(\mathbb{R})$ and $T: \rho_0 \di x \to \rho_1 \di y$ is $W_2$-optimal transport (Bernier's map). Then 
$\rho_t:= [tI + (1-t)T ]_\# \rho_0$
is a Wasserstein geodesics (or displacement interpolation between $\rho_0$ and $\rho_1$). From the definition of push forward (see \cite[Section 5.2]{Ambrosio}), 
\begin{align*}
E_{\textnormal{h}}(\rho_t)=&\gamma\int_{\mathbb{R}}\frac{x^2}{2}\rho_t(\di x)= \gamma\int_{\mathbb{R}} \frac{[t x +(1-t)T(x)]^2}{2} \rho_0(\di x) \\
=& \gamma\int_{\mathbb{R}} \frac{t x^2 -t(1-t)(x-T(x))^2+(1-t)T^2(x)}{2} \rho_0(\di x)\\
=& tE_{\textnormal{h}}(\rho_0)+(1-t) E_{\textnormal{h}}(\rho_1) - \gamma\frac{t(1-t)}{2} W_2^2(\rho_0, \rho_1).
\end{align*}
Therefore $E_{\textnormal{h}}(\rho)$ is $\gamma$-geodesically convex  (see  \cite[Definition 2.4.3]{Ambrosio}).  For the geodesical convexity of the interaction energy $E_{\textnormal{i}}(\rho)$, due to the singularity in logarithmic function, it relies heavily on monotonicity of optimal map.  We illustrate the idea for $\rho_0, \rho_1\in \mathcal{P}_{\text{AC}}(\mathbb{R}), \rho_0>0$, which ensures the optimal map $T$ is strictly increasing.
\begin{align*}
E_i(\rho_t)=& \int_{\mathbb{R}^2} - \log(|x-y|) \rho_t(\di x ) \rho_t(\di y) \\
=& \int_{\mathbb{R}^2} - \log (|t(x-y)+(1-t)(T(x)-T(y))|)\rho_0(x) \rho_0(y) \di x \di y\\
\leq & t \int_{\mathbb{R}^2} - \log|x-y| \rho_0(x) \rho_0(y) \di x \di y + (1-t)  \int_{\mathbb{R}^2} - \log|T(x)-T(y)| \rho_0(x) \rho_0(y) \di x \di y \\
= & t \int_{\mathbb{R}^2} - \log|x-y| \rho_0(x) \rho_0(y) \di x \di y + (1-t)  \int_{\mathbb{R}^2} - \log|x-y| \rho_1(x) \rho_1(y) \di x \di y \\
=& t E_i(\rho_0)+(1-t) E_i(\rho_1),
\end{align*}
where we used the convexity of logarithmic function in the first inequality and strict increase of $T$ in the third equality.
However, without the strictly  increasing property, we refer to \cite[Proposition 2.7]{carrillo2012mass}, where Carrillo et. al. proved the generalized geodesic convex of $E_i(\rho)$ using the essential monotonicity property (excluding a null set) of the optimal transport maps between absolutely continuous probability measures in one dimension.  The standard gradient flow theory \cite[Theorem 11.2.1]{Ambrosio} yields the exponential convergence to the steady state in $W_2$ distance; see also \cite{carrillo2012mass, berman2019propagation}. More precisely, if $\rho$ and $\tilde{\rho}$ are two probability measure solutions for initial date $\rho_0$ and $\tilde{\rho}_0$ separately, then we have
\begin{align}\label{eq:contraction}
W_2(\rho(t),\tilde{\rho}(t))\leq e^{-\gamma t}W_2(\rho_0,\tilde{\rho}_0).
\end{align}
This implies the uniqueness of probability measure solutions and exponential convergence to the steady state.

When $\gamma>0$, we also remark that $\gamma$-convexity of $E$ implies the uniqueness of the steady state (minimizer). Indeed, if $\mu$ and $\nu$ are two distinct minimizers, consider $\mu_{1/2}:= [\frac{1}{2}I + \frac{1}{2}\tilde{T} ]_\# \mu$, where $\tilde{T}$ is Bernier's map between $\mu$ and $\nu$. Then, we have
\[
E(\mu_{1/2})\leq \frac{1}{2}[E(\mu)+E(\nu)]-\frac{\gamma}{8} W_2^2(\mu, \nu)<\frac{1}{2}[E(\mu)+E(\nu)],
\]
which is a contradiction with that $\mu$ and $\nu$ are distinct.
\end{remark}

\section{Bi-Hamiltonian structures}\label{sec:bihamiltonian}
In this section, we construct a bi-Hamiltonian structure for the coupled Burgers system \eqref{eq:systemEuler1} by using the decoupled Burgers equations \eqref{eq:decouple1} and \eqref{eq:decouple2}. 
First, we present infinite many conserved quantities for the coupled Burgers system  \eqref{eq:systemEuler1}. Recall \eqref{fpm}.
Because $\int_{\mathbb{R}}f_{\pm}^k(x,t)\,\di x$ are conserved quantities of the decoupled Burgers equations \eqref{eq:decouple1} and \eqref{eq:decouple2}, we have the following proposition.
\begin{proposition}
Let $(\rho,u)$ be a classical solution to the coupled Burgers system  \eqref{eq:systemEuler1}. Then, quantities
\begin{align}\label{eq:conservedq}
\lambda_1\int_{\mathbb{R}}(u+\sqrt{\alpha}\rho)^{k_1}\di x+\lambda_2\int_{\mathbb{R}}(u-\sqrt{\alpha}\rho)^{k_2} \di x
\end{align}
are conserved for any constants $\lambda_1,\lambda_2\in\mathbb{C}$ and any positive integers $k_1,~k_2\in \mathbb{N}_+$.
\end{proposition}

\begin{remark}\label{rmk:realquantities}
Notice that when $\alpha<0$, we have $f_-=\bar{f}_+$. When  $\lambda_1=\bar{\lambda}_2$ and $k_1=k_2=k$ in \eqref{eq:conservedq}, we have
\[
\overline{\lambda_1(u+\sqrt{\alpha}\rho)^k}=\lambda_2(u-\sqrt{\alpha}\rho)^k.
\]
In this case, \eqref{eq:conservedq} gives real conserved quantities.
\end{remark}
Next, we consider the case  for $k_1=k_2=3$ in \eqref{eq:conservedq} and derive a bi-Hamiltonian structure for the coupled Burgers system \eqref{eq:systemEuler1}. Define the following functionals of $f_\pm (=u\pm\sqrt{\alpha}\rho)$:
\begin{align}\label{eq:conservedf}
H_1^f(f_+,f_-):=\int_{\mathbb{R}}\frac{f_+^3 + f_-^3}{12}\di x,\quad  H_2^f(f_+,f_-):=\int_{\mathbb{R}}\frac{f_+^3 - f_-^3}{12\sqrt{\alpha}}\di x.
\end{align}
Due to Remark \ref{rmk:realquantities}, we know that both $H_1^f$ and $H_2^f$ are  real conserved quantities.
Moreover, the decoupled Burgers equations \eqref{eq:decouple1} and \eqref{eq:decouple2} can be rewritten as
\begin{gather}\label{eq:decoupled2}
\left\{
\begin{split}
&\partial_tf_++2\partial_x\left(\frac{\delta{H^f_1}}{\delta f_+}\right)=0,\\
&\partial_tf_-+2\partial_x\left(\frac{\delta{H^f_1}}{\delta f_-}\right)=0,
\end{split}
\right.~~\textrm{ and }~~ \left\{
\begin{split}
&\partial_tf_++2\sqrt{\alpha}\partial_x\left(\frac{\delta{H^f_2}}{\delta f_+}\right)=0,\\
&\partial_tf_--2\sqrt{\alpha}\partial_x\left(\frac{\delta{H^f_2}}{\delta f_-}\right)=0.
\end{split}
\right.
\end{gather}
Define 
\begin{align}\label{eq:generalAlpha}
H^u_1(\rho,u):=\int_{\mathbb{R}}\left(\frac{1}{6}u^3+\frac{\alpha}{2}\rho^2 u\right) \di x,\quad H^u_2(\rho,u):=\int_{\mathbb{R}}\left(\frac{1}{2}\rho u^2+\frac{\alpha}{6}\rho^3\right) \di x.
\end{align}
Then, direct calculations show that 
\begin{align}\label{eq:relation}
H^{u}_j(\rho,u)=H^{f}_j(f_+,f_-),\quad j=1,2,
\end{align}
and we have the following theorem:
\begin{theorem}\label{thm:biHamiltonian}
For $\alpha\neq0$, the coupled Burgers system \eqref{eq:systemEuler1} has a bi-Hamiltonian structure:
\begin{equation}\label{eq:EulerHamiltonian}
\frac{\partial}{\partial t}
\begin{pmatrix}
\rho\\
u
\end{pmatrix}
=
J
\begin{pmatrix}
\frac{\delta H^{u}_1}{\delta \rho} \\ \frac{\delta H^{u}_1}{\delta u}
\end{pmatrix}=
K
\begin{pmatrix}
\frac{\delta H^{u}_2}{\delta \rho} \\ \frac{\delta H^{u}_2}{\delta u}
\end{pmatrix},
\end{equation}
where $J$ and $K$ are anti-symmetric operators given by
\begin{align}\label{eq:operators1}
J:=\begin{pmatrix}
-\frac{1}{\alpha}\partial_x  & 0\\
0  & -\partial_x
\end{pmatrix},
\quad K:=\begin{pmatrix}
0 & -\partial_x \\
-\partial_x   & 0
\end{pmatrix}.
\end{align}

\end{theorem}

\begin{proof}
Due to $f_\pm=u\pm \sqrt{\alpha}\rho$, we have $\rho =\frac{1}{2\sqrt{\alpha}}(f_+-f_-)$ and $u=\frac{1}{2}(f_++f_-)$. From \eqref{eq:decoupled2}, we obtain
\begin{gather}\label{eq:Burgersvariation}
\left\{
\begin{split}
&\partial_t\rho+\frac{1}{\sqrt{\alpha}}\partial_x\left(\frac{\delta{H^f_1}}{\delta f_+}-\frac{\delta{H^f_1}}{\delta f_-}\right)=0,\\
&\partial_tu+\partial_x\left(\frac{\delta{H^f_1}}{\delta f_+}+\frac{\delta{H^f_1}}{\delta f_-}\right)=0,
\end{split}
\right.~~\textrm{ and }~~
\left\{
\begin{split}
&\partial_t\rho+\partial_x\left(\frac{\delta{H^f_2}}{\delta f_+}+\frac{\delta{H^f_2}}{\delta f_-}\right)=0,\\
&\partial_tu+\sqrt{\alpha}\partial_x\left(\frac{\delta{H^f_2}}{\delta f_+}-\frac{\delta{H^f_2}}{\delta f_-}\right)=0.
\end{split}
\right.
\end{gather}
Due to \eqref{eq:relation}, we have the following relations:
\begin{align}\label{eq:variationrelations1}
\frac{\delta H^{u}_j}{\delta \rho}=\sqrt{\alpha}\left(\frac{\delta H^{f}_j}{\delta f_+}-\frac{\delta H^{f}_j}{\delta f_-}\right),\quad \frac{\delta H^{u}_j}{\delta u}=\frac{\delta H^{f}_j}{\delta f_+}+\frac{\delta H^{f}_j}{\delta f_-},\quad j=1,2.
\end{align}
Put \eqref{eq:variationrelations1} into \eqref{eq:Burgersvariation} and we obtain
\begin{gather*}
\left\{
\begin{split}
&\partial_t\rho+\frac{1}{\alpha}\partial_x\left(\frac{\delta H^{u}_1}{\delta \rho}\right)=0,\\
&\partial_tu+\partial_x\left(\frac{\delta H^{u}_1}{\delta u}\right)=0,
\end{split}
\right.~\textrm{ and }~
\left\{
\begin{split}
&\partial_t\rho+\partial_x\left(\frac{\delta H^{u}_2}{\delta u}\right)=0,\\
&\partial_tu+\partial_x\left(\frac{\delta H^{u}_2}{\delta \rho}\right)=0,
\end{split}
\right.
\end{gather*}
which is \eqref{eq:EulerHamiltonian}.
\end{proof}

From Theorem \ref{thm:biHamiltonian}, we can directly obtain a bi-Hamiltonian structure for System \eqref{eq:systemEuler3}, as shown in the following corollary:
\begin{corollary}
For $\alpha\neq0$, the isentropic gas dynamics \eqref{eq:systemEuler3} can be rewritten as the following bi-Hamiltonian structure:
\begin{equation}\label{eq:EulerHamiltonian3}
\frac{\partial}{\partial t}
\begin{pmatrix}
\rho\\
m
\end{pmatrix}
=
\tilde{J}
\begin{pmatrix}
\frac{\delta H^{m}_1}{\delta \rho} \\ \frac{\delta H^{m}_1}{\delta m}
\end{pmatrix}=
\tilde{K}
\begin{pmatrix}
\frac{\delta H^{m}_2}{\delta \rho} \\ \frac{\delta H^{m}_2}{\delta m}
\end{pmatrix},
\end{equation}
where $\tilde{J}$ and $\tilde{K}$ are anti-symmetric operators given by
\begin{align}\label{eq:operators2}
\tilde{J}=&
\begin{pmatrix}
-\frac{1}{\alpha}\partial_x  & -\frac{1}{\alpha}\partial_xu\\
-\frac{1}{\alpha}u\partial_x  & -\frac{1}{\alpha}u\partial_xu-\rho\partial_x \rho
\end{pmatrix},
\quad \tilde{K}=\begin{pmatrix}
0  & -\partial_x\rho\\
-\rho\partial_x  & -u\partial_x\rho-\rho\partial_x u
\end{pmatrix},
\end{align}
and the Hamiltonians are given by
\begin{align}\label{eq:mrho}
H^m_1(\rho,m)=\int_{\mathbb{R}}\left(\frac{m^3}{6\rho^3}+\frac{\alpha}{2}m\rho\right) \di x,\quad H^m_2(\rho,m)=\int_{\mathbb{R}}\left(\frac{m^2}{2\rho}+\frac{\alpha}{6}\rho^3\right) \di x.
\end{align}
\end{corollary}

\begin{proof}
Due to $m=\rho u$, we have
\[
H^{m}_j(\rho,m):=H^{u}_j(\rho,u),~~j=1,2.
\]
Moreover, we have the following relations:
\[
\frac{\delta {H}^{u}_j}{\delta \rho}=\frac{\delta H^{m}_j}{\delta \rho}+u\frac{\delta H^{m}_j}{\delta m},\quad \frac{\delta {H}^{u}_j}{\delta u}=\rho \frac{\delta H^{m}_j}{\delta m},
\]
or equivalently
\begin{equation}\label{eq:variationrelations3}
\begin{pmatrix}
\frac{\delta {H}^{u}_j}{\delta \rho} \\ \frac{\delta {H}^{u}_j}{\delta u}
\end{pmatrix}
=
\begin{pmatrix}
1  & u\\
0  & \rho
\end{pmatrix}
\begin{pmatrix}
\frac{\delta H^{m}_j}{\delta \rho} \\ \frac{\delta H^{m}_j}{\delta m}
\end{pmatrix},~~j=1,2.
\end{equation}
Combining \eqref{eq:EulerHamiltonian} and \eqref{eq:variationrelations3},  we obtain
\begin{align}\label{eq:mEulerHamiltonian1}
\frac{\partial}{\partial t}
\begin{pmatrix}
\rho\\
m
\end{pmatrix}
=&
\begin{pmatrix}
1  & 0\\
u  & \rho
\end{pmatrix}
\begin{pmatrix}
\partial_t\rho\\
\partial_t u
\end{pmatrix}
=\begin{pmatrix}
1 & 0\\
u & \rho 
\end{pmatrix}
J
\begin{pmatrix}
1  & u\\
0  & \rho
\end{pmatrix}
\begin{pmatrix}
\frac{\delta H^{m}_1}{\delta \rho} \\ \frac{\delta H^{m}_1}{\delta m}
\end{pmatrix}\nonumber\\
&=\begin{pmatrix}
1  & 0\\
u  & \rho
\end{pmatrix}
K
\begin{pmatrix}
1  & u\\
0  & \rho
\end{pmatrix}
\begin{pmatrix}
\frac{\delta H^{m}_2}{\delta \rho} \\ \frac{\delta H^{m}_2}{\delta m}
\end{pmatrix}.
\end{align}
Hence, we have
\begin{align*}
\tilde{J}=\begin{pmatrix}
1 & 0\\
u & \rho 
\end{pmatrix}
J
\begin{pmatrix}
1  & u\\
0  & \rho
\end{pmatrix}
=
\begin{pmatrix}
-\frac{1}{\alpha}\partial_x  & -\frac{1}{\alpha}\partial_xu\\
-\frac{1}{\alpha}u\partial_x  & -\frac{1}{\alpha}u\partial_xu-\rho\partial_x \rho
\end{pmatrix},
\end{align*}
and
\begin{align*}
\tilde{K}=\begin{pmatrix}
1  & 0\\
u  & \rho
\end{pmatrix}
K
\begin{pmatrix}
1  & u\\
0  & \rho
\end{pmatrix}
=
\begin{pmatrix}
0  & -\partial_x\rho \\
-\rho\partial_x & -u\partial_x\rho-\rho\partial_x u
\end{pmatrix}.
\end{align*}
Hence, we obtain a bi-Hamiltonian structure for System \eqref{eq:systemEuler3}.
\end{proof}

Notice that $H_2^m$ is nothing but the total energy of System \eqref{eq:systemEuler3}, which is given by
\begin{align}\label{eq:Hamilton1}
H^m_2(\rho,m)=\int_{\mathbb{R}}E(x,t)\di x=\int_{\mathbb{R}}\left(\frac{1}{2}\rho u^2+\frac{\alpha }{6}\rho^3\right)\di x=\int_{\mathbb{R}}\left(\frac{m^2}{2\rho}+\frac{\alpha \rho^3}{6}\right)\di x.
\end{align}
where $E(x,t)$ is defined by \eqref{eq:energyflux}.

\begin{remark}[A bi-Hamiltonian structure for p-system \eqref{eq:gas}]
Set
\[
\eta(\xi,t):=\frac{1}{\tau(\xi,t)},\quad \xi\in(0,1),~~t>0.
\]
Then, the p-system \eqref{eq:gas} becomes the following system for $(\eta, ~V):$
\begin{gather}\label{eq:systemLagrange}
\left\{
\begin{split}
&\partial_\xi\eta=-\eta^2\partial_\xi V,\\
&\partial_tV=-\alpha\eta^2\partial_\xi\eta.
\end{split}
\right.
\end{gather}
We have the following bi-Hamiltonian structure for System \eqref{eq:systemLagrange}:
\begin{equation*}
\frac{\partial}{\partial t}
\begin{pmatrix}
\eta\\
V
\end{pmatrix}
=
\begin{pmatrix}
-\frac{3}{4\alpha}\eta\partial_\xi\eta  & -\frac{1}{4\alpha}\eta\partial_\xi V+\frac{3}{2\alpha}V\partial_\xi \eta-\frac{1}{\alpha}\partial_\xi \eta V\\
-\frac{1}{4\alpha}V\partial_\xi \eta +\frac{3}{2\alpha}\eta \partial_\xi V-\frac{1}{\alpha}\eta V\partial_\xi  & \frac{1}{4\alpha}V\partial_\xi V+\frac{1}{\alpha}\eta \partial_\xi\eta
\end{pmatrix}
\begin{pmatrix}
\frac{\delta H^\eta_1}{\delta \eta} \\ \frac{\delta H^\eta_1}{\delta V}
\end{pmatrix},
\end{equation*}
and
\begin{equation*}
\frac{\partial}{\partial t}
\begin{pmatrix}
\eta\\
V
\end{pmatrix}
=
\begin{pmatrix}
0  & -\eta^2\partial_\xi\\
-\partial_\xi \eta^2 & 0
\end{pmatrix}
\begin{pmatrix}
\frac{\delta H^\eta_2}{\delta \eta} \\ \frac{\delta H^\eta_2}{\delta V}
\end{pmatrix},
\end{equation*}
where
\begin{align*}
H^\eta_1(\eta,V)=\int_{\mathbb{R}}\left(\frac{V^3}{6\eta}+\alpha\frac{\eta V}{2}\right)\di \xi,\quad H^\eta_2(\eta,V)=\int_{\mathbb{R}}\left(\frac{V^2}{2}+\alpha\frac{\eta^2}{6}\right)\di \xi.
\end{align*}
\end{remark}

\section{Kinetic formulations and entropy solutions for the coupled Burgers system \eqref{eq:systemEuler1} with $\alpha>0$}\label{sec:kinetic}
In this section, we study the kinetic formulation for the coupled Burgers system \eqref{eq:systemEuler1} with $\alpha>0$. In contrast, Lions, Perthame and Tadmor \cite{Lions1994Kinetic}  studied System \eqref{eq:systemEuler3} and they used the kinetic formulation to obtain global entropy solutions without uniqueness.
Here, we show the existence and uniqueness of global  entropy solutions for  \eqref{eq:systemEuler1}.

\subsection{Kinetic formulations}
 Kinetic formulation is a method which use the distribution function $\kappa(v,x,t)$ at time $t$ in the phase plane  for velocity $v$ and the position $x$  to study the continuum equation for $u(x,t)$ (and $\rho(x,t)$).  At fixed continnum variable $(x,t)$, $u$ and $\rho$ are some $v$-moments of $\kappa$. In the local thermal equilibrium the distribution function $\kappa(v,x,t)$ can be described by $v$-equilibrium distribution $\chi(v; u, \rho)$ with parameters  $u$ and $\rho$, i.e. $\kappa(v,x,t)=\chi(v; u(x,t), \rho(x,t))$. In kinetic theory, the $v$-equilibrium distribution is also known as Maxwellian. Following the idea of the celebrated work by Lions, Perthame and Tadmor \cite{Lions1994Kinetic}, we use the combinations of Heaviside function,
\begin{gather*}
H(v)=\left\{
\begin{split}
1,\quad v\geq0,\\
0,\quad v<0,
\end{split}
\right.
\end{gather*}
 to construct the equilibrium distribution. 
Let $(\rho,u)$ be a solution  to the coupled Burgers system \eqref{eq:systemEuler1} with $\alpha>0$. Recall \eqref{fpm}
\begin{align*}
f_\pm = u \pm \sqrt{\alpha}\rho.
\end{align*}
Then, $f_\pm$ are solutions to the decoupled Burgers equations \eqref{eq:decouple1} and \eqref{eq:decouple2}. 
We use the following   $v$-equilibrium distributions
\[
\chi_+(v;\rho,u):=H(v)-H(v-f_+),\quad \chi_-(v;\rho,u):=H(v)-H(v-f_-),
\]
\begin{equation}\label{eq:kineticf}
\chi(v;\rho,u):=\frac{1}{2\sqrt{\alpha}}(\chi_+ - \chi_-)=\frac{1}{2\sqrt{\alpha}}[H(v-f_-)-H(v-f_+)],
\end{equation}
and
\[
\hat{\chi}(v;\rho,u):=\frac{1}{2}(\chi_+ + \chi_-)=\frac{1}{2}[2H(v)-H(v-f_-)-H(v-f_+)].
\]
For any nonnegative integer $k$, direct calculations show that the following $k$-moments equality holds
\begin{align}\label{eq:kineticrepDB}
\int_{\mathbb{R}}v^k\chi_\pm (v;\rho,u) \di v = \frac{f^{k+1}_\pm}{k+1}.
\end{align}
Hence, the conserved quantities given  by \eqref{eq:conservedq} correspond to the integration (w.r.t. $x$ variable) of the following kinetic formulations:
\begin{align}\label{eq:Hamiltonaninre}
\lambda_1\int_{\mathbb{R}}v^{k_1}\chi_+ (v;\rho,u) \di v + \lambda_2\int_{\mathbb{R}}v^{k_2}\chi_- (v;\rho,u) \di v,\quad \lambda_i\in\mathbb{C},~k_i\in \mathbb{N}_+,~~i=1,2.
\end{align}
Choosing $\lambda_1=\lambda_2=\frac{1}{4}$ and $k_1=k_2=2$, we obtain the Hamiltonian $H_1^u$ and choosing  $\lambda_1=-\lambda_2=\frac{1}{4\sqrt{\alpha}}$ and $k_1=k_2=2$, we obtain the Hamiltonian $H_2^u$ given by \eqref{eq:generalAlpha}. More precisely, we have
\begin{align}\label{eq:Hamiltonianu12}
H_1^u(\rho,u)=\frac{1}{2}\int_{\mathbb{R}}v^2\hat{\chi} (v;\rho,u) \di v,\quad H_2^u(\rho,u)=\frac{1}{2}\int_{\mathbb{R}}v^2\chi (v;\rho,u) \di v.
\end{align}
By \eqref{eq:kineticrepDB}, the decoupled Burgers equations \eqref{eq:decouple1} and \eqref{eq:decouple2} have the following kinetic formulations:
\[
\int_{\mathbb{R}}(\partial_t\chi_\pm + v\partial_x\chi_\pm)\di v=\partial_tf_\pm  + f_\pm \partial_x f_\pm = 0.
\]
Besides, we also have
\[
u=\int_{\mathbb{R}}\hat{\chi}(v;\rho,u)\di v,\quad \rho=\int_{\mathbb{R}}\chi(v;\rho,u)\di v.
\]
Hence, the coupled Burgers system \eqref{eq:systemEuler1} has the following kinetic formulation:
\begin{equation}\label{eq:kineticcoupled}
\int_{\mathbb{R}} \begin{pmatrix}
\partial_t{\chi}+v\partial_x{\chi}\\
\partial_t\hat{\chi}+v\partial_x\hat{\chi}
\end{pmatrix}\di v
=
\begin{pmatrix}
\partial_t\rho+\partial_x(\rho u) \\
\partial_tu+\partial_x\left(\frac{u^2+\alpha \rho^2}{2}\right)
\end{pmatrix}
=\begin{pmatrix}
0\\
0
\end{pmatrix}.
\end{equation}
Moreover, direct calculations show that
\begin{equation}\label{eq:kineticrepSE}
\begin{pmatrix}
\rho u\\
E
\end{pmatrix}
=
\int_{\mathbb{R}}
\begin{pmatrix}
v \\ \frac{v^2}{2}
\end{pmatrix} \chi(v;\rho, u)\di v,
\end{equation}
where $E$ is the total energy given by \eqref{eq:energyflux}. Comparing with \eqref{eq:kineticcoupled}, we have the following kinetic formulation for the isentropic gas system \eqref{eq:systemEuler3}:
\begin{equation*}
\int_{\mathbb{R}} \begin{pmatrix}
1\\
v
\end{pmatrix}(\partial_t\chi+v\partial_x\chi)\di v
=
\begin{pmatrix}
\partial_t\rho+\partial_x(\rho u) \\
\partial_t(\rho u)+\partial_x(\rho u^2+ p)
\end{pmatrix}
=\begin{pmatrix}
0\\
0
\end{pmatrix}.
\end{equation*}

\subsection{Existence and uniqueness of entropy solutions for the coupled Burgers system \eqref{eq:systemEuler1}}\label{sec:entropy}
The notion of the entropy-entropy-flux pair refers to the pair of regular functions $(\eta,q)$ defined on the space of the states $(\rho,u)$ for which every classical solution $(\rho,u)$ of the coupled Burgers system \eqref{eq:systemEuler1},  also satisfies
\begin{align}\label{eq:entropyequality}
\partial_t\eta(\rho,u)+\partial_x q(\rho,u)=0.
\end{align}
Combining the coupled Burgers system \eqref{eq:systemEuler1} and \eqref{eq:entropyequality} gives
\[
(\partial_\rho q - u \partial_\rho\eta - \alpha \rho \partial_u\eta)\partial_x\rho+(\partial_uq - \rho \partial_\rho\eta - u\partial_u\eta)\partial_xu=0,
\]
which holds for any smooth solutions $(\rho ,u)$. The entropy pair $(\eta, q)$ can be found by solving the following Euler-Poisson-Darboux equations.  First, given $\psi(u), g(u)$, 
   we solve $\eta(\rho, u)$ satisfying
\begin{gather}\label{eq:waveequation}
\left\{
\begin{split}
&\partial_{\rho\rho}\eta-\alpha \partial_{uu}\eta=0,\\
&\eta(0,u)=\psi(u),~~\partial_{\rho}\eta(0,u)=g(u).
\end{split}
\right.
\end{gather}
Then we solve the  entropy flux $q$  by
\begin{align}\label{eq:entropyflux}
\partial_uq = \rho \partial_\rho\eta + u\partial_u\eta,\quad \partial_\rho q = u \partial_\rho\eta + \alpha \rho \partial_u\eta.
\end{align}
From \eqref{eq:waveequation}, we know $\partial_{u\rho}q=\partial_{\rho u}q$ so \eqref{eq:entropyflux} is solvable.
We have the following results:
\begin{proposition}\label{pro:entropy1}
For two given functions $\psi,g\in C^2(\mathbb{R})$, let $(\eta(0,u),\partial_\rho\eta(0,u))=(\psi(u),g(u))$ be the initial datum for \eqref{eq:waveequation}. Then:

$\mathrm{(i)}$ The  solution $\eta(\rho,u)$ to  \eqref{eq:waveequation}  can be recast in a kinetic representation:
\begin{align}\label{eq:dAlembertkinetic}
\eta(\rho,u)&=\int_{\mathbb{R}}\psi'(v)\hat{\chi}(v;\rho,u)\di v+\int_{\mathbb{R}}g(v){\chi}(v;\rho,u)\di v
\end{align}

$\mathrm{(ii)}$ When $\psi=0$ and $\rho\geq0$, we have kinetic representations 
\begin{align}\label{eq:entropyfluxg}
\eta_g(\rho, u):=\int_{\mathbb{R}}g(v){\chi}(v;\rho,u)\di v , \quad q_g(\rho,u):=\int_{\mathbb{R}} vg(v)\chi(v;\rho,u)\di v.
\end{align}
Moreover, $\eta_g$
 is convex with respect to $(\rho,m)$ if and only if $g(v)$ is convex, where $m=\rho u$.

$\mathrm{(iii)}$ When $g=0$, 
we have kinetic representations 
\begin{align}\label{eq:kineticflux}
\eta_\psi(\rho,u) := \int_{\mathbb{R}}\psi'(v)\hat{\chi}(v;\rho,u)\di v\\
q_\psi(\rho,u):=\int_{\mathbb{R}}v\psi'(v)\hat{\chi}(v;\rho,u)\di v=\frac{\phi(u+\sqrt{\alpha}\rho)+\phi(u-\sqrt{\alpha}\rho)}{2},
\end{align}
where $\phi'(v)=v\psi'(v)~\textrm{ for }~v\in\mathbb{R}.$ 
Moreover,
$\eta_\psi$  is convex with respect to $(\rho,u)$ if and only if $\psi$ is a convex function. 
\end{proposition}

\begin{proof}
(i) By the d'Alembert's formula, we have
\begin{align}\label{eq:dAlembert}
\eta(\rho,u)=\frac{\psi(f_+)+\psi(f_-)}{2}+\frac{1}{2\sqrt{\alpha}}\int_{f_-}^{f_+}g(v)\di v,
\end{align}
where $f_\pm(\rho,u)=u\pm\sqrt{\alpha}\rho$. Formula \eqref{eq:dAlembertkinetic}  is exactly the kinetic formulation for the formula \eqref{eq:dAlembert}.

(ii)  
First we verify \eqref{eq:entropyfluxg} satisfies \eqref{eq:entropyflux}. We have
\[
\partial_u\eta_g = g(u+\sqrt{\alpha}\rho)-g(u-\sqrt{\alpha}\rho),\quad \partial_\rho\eta_g = \sqrt{\alpha}\Big(g(u+\sqrt{\alpha}\rho)+g(u-\sqrt{\alpha}\rho)\Big),
\]
and
\begin{align*}
\partial_uq_g=&(u+\sqrt{\alpha}\rho)g(u+\sqrt{\alpha}\rho)-(u-\sqrt{\alpha}\rho)g(u-\sqrt{\alpha}\rho)\\
=&u\partial_u\eta_g+\rho\partial_\rho\eta_g.
\end{align*}
Similarly, we also have $\partial_\rho q_\psi=u\partial_\rho\eta_\psi+\alpha\rho\partial_u\eta_\psi$.

Second, we check the  convexity condition for $\eta_g$ in terms of $(\rho, m)$, where $m=\rho u$.
By changing of variables $v=u\pm \xi \rho$, we have
\begin{align}\label{eq:entropyg}
\eta_g(\rho,u)=&\int_{\mathbb{R}}g(v){\chi}(v;\rho,u)\di v=\frac{1}{2\sqrt{\alpha}}\int_{f_-}^{f_+}g(v)\di v\nonumber\\
=&\frac{1}{2\sqrt{\alpha}}\int_{-\sqrt{\alpha}}^{\sqrt{\alpha}}\rho g\left(\frac{m}{\rho}+\xi\rho\right)\di \xi.
\end{align}
Taking derivatives of \eqref{eq:entropyg}, we can obtain
\[
\partial_{\rho\rho}\eta_g=\frac{1}{2\sqrt{\alpha}}\int_{-\sqrt{\alpha}}^{\sqrt{\alpha}}\rho g''\left(\frac{m}{\rho}+\xi\rho\right)\left(-\frac{m}{\rho^2}+\xi\right)^2 \di \xi+\frac{1}{\sqrt{\alpha}}\int_{0}^{\sqrt{\alpha}}\xi\Big[ g'(u+\xi\rho)-g'(u-\xi\rho)\Big] \di \xi.
\]
When $g$ is convex, $g'$ is increasing and $g''>0$. Hence $\partial_{\rho\rho}\eta_g\geq0$. Moreover, we have
\[
\partial_{\rho m}\eta_g=\frac{1}{2\sqrt{\alpha}}\int_{-\sqrt{\alpha}}^{\sqrt{\alpha}}g''\left(\frac{m}{\rho}+\xi\rho\right)\left(-\frac{m}{\rho^2}+\xi\right) \di \xi,
\]
and
\[
\partial_{mm}\eta_g=\frac{1}{2\sqrt{\alpha}}\int_{-\sqrt{\alpha}}^{\sqrt{\alpha}}\frac{1}{\rho}g''\left(\frac{m}{\rho}+\xi\rho\right)\di \xi\geq0.
\]
By H\"older's inequality, we can obtain
\[
\partial_{\rho\rho}\eta_g\cdot \partial_{mm}\eta_g-(\partial_{\rho m}\eta_g)^2\geq 0.
\]
Hence, $\eta_g(\rho,u)$ is convex about $(\rho,m)$. When $g$ is not convex, we have $g''<0$ in some interval. This implies $\partial_{mm}\eta_g<0$. Hence, $\eta_g$ is not convex. This proves that $\eta_g$ is convex if and only if $g$ is convex.

(iii) First, we verify equalities in \eqref{eq:entropyflux} hold for $(\eta_\psi,q_\psi)$.
For $f_\pm=u\pm\sqrt{\alpha}\rho$, we have
\begin{align*}
\partial_uq_\psi=\frac{f_+\psi'(f_+)+f_-\psi'(f_-)}{2}=&u\frac{\psi'(f_+)+\psi'(f_-)}{2}+\rho\frac{\sqrt{\alpha}\psi'(f_+)-\sqrt{\alpha}\psi'(f_-)}{2}\\
=&u\partial_u\eta_\psi+\rho\partial_\rho\eta_\psi.
\end{align*}
Similarly, $\partial_\rho q_\psi=u\partial_\rho\eta_\psi+\alpha\rho\partial_u\eta_\psi.$
Hence, equalities in \eqref{eq:entropyflux} hold for $(\eta_\psi,q_\psi)$. This proves that $q_\psi$ is the corresponding entropy flux of $\eta_\psi$.

Second, we check the  convexity condition for $\eta_{\psi}$ in terms of $(\rho, u)$.
Notice that
\begin{align}\label{eq:dAlembert1}
\eta_\psi(\rho,u)=\int_{\mathbb{R}}\psi'(v)\hat{\chi}(v;\rho,u)\di v=\frac{\psi(f_+)+\psi(f_-)}{2},
\end{align}
where $f_\pm=u\pm\sqrt{\alpha}\rho$.
Taking derivative of \eqref{eq:dAlembert1}, we can obtain
\[
\partial_{\rho u}\eta_\psi=\frac{\sqrt{\alpha}(\psi''(f_+)-\psi''(f_-))}{2},\quad \partial_{\rho \rho}\eta_\psi=\frac{{\alpha}(\psi''(f_+)+\psi''(f_-))}{2},
\]
and
\[
\partial_{u u}\eta_\psi=\frac{\psi''(f_+)+\psi''(f_-)}{2}.
\]
When $\psi$ is convex, we have
$$\partial_{u u}\eta_\psi\geq0,\partial_{u u}\eta_\psi\geq0,~\textrm{ and }~\partial_{\rho \rho}\eta_\psi\partial_{u u}\eta_\psi\geq (\partial_{\rho u}\eta_\psi)^2,$$
which means $\eta_\psi$ is convex with respect to $(\rho,u)$. Conversely, if $\eta_\psi$ is convex with respect to $(\rho,u)$, $\psi(u)=\eta_\psi(0,u)$ is convex.

\end{proof}

In \cite{Lions1994Kinetic}, Lions, Perthame and Tadmor studied the kinetic formulation of the isentropic gas system \eqref{eq:systemEuler3}. The convex entropies they used to define solutions corresponds to $\eta_g(\rho,u)$ given by \eqref{eq:dAlembertkinetic} for convex functions $g$. The corresponding entropy flux are given by \eqref{eq:entropyfluxg}. Recall their definition of the entropy solutions to the isentropic gas system  \eqref{eq:systemEuler3}  (see \cite[Definition 2]{Lions1994Kinetic}).
\begin{definition}\label{def:entropysolution}
A couple $(\rho,m)$ is called an entropy solution of \eqref{eq:systemEuler3} if it satisfies
\begin{align}\label{eq:entropyIneq1}
\partial_t\eta_g(\rho,u)+\partial_xq_g(\rho,u)\leq0,
\end{align}
in distribution sense for all convex entropies $\eta_g$ given by \eqref{eq:dAlembertkinetic} with convex $g$.
\end{definition}
An important example for  is taking $g(v)=\frac{v^2}{2}$ in \eqref{eq:entropyg}.
Direct calculations show that
the entropy has the following kinetic formulation
\begin{equation*}
\eta_g=\frac{1}{2} \rho u^2 + \frac{\alpha }{6}\rho^3=: E
\end{equation*}
 and
 the entropic flux is 
\[
q_g=\int_{\mathbb{R}} vg(v) \chi(v;\rho,u)\di v=\frac{1}{4\sqrt{\alpha}}\int_{f_-}^{f_+}v^3\di v=\frac{1}{16\sqrt{\alpha}}(f_+^4-f_-^4)=u(E+p).
\]
Then  \eqref{eq:entropyIneq1} in  Definition \ref{def:entropysolution} becomes
\begin{align}\label{eq:energyentropy}
\partial_tE+\partial_x[(E+p)u]\leq 0
\end{align}
in the distributional sense.
Notice that $g(v)=\frac{v^2}{2}$ is convex and hence $E$ is convex with respect to $(\rho,m)$.

\begin{remark}
Note that global existence of entropy solutions to System \eqref{eq:systemEuler3} was proved \cite{Lions1994Kinetic}.
It is shown in \cite{Lions1994Kinetic} that $(\rho,m)$ is a weak entropy solution with respect to the family $\{\eta_g \}$, if and only if the kinetic function $\chi(v;\rho,u)$ given by \eqref{eq:kineticf} is a weak solution of the kinetic equation
\[
\partial_t\chi+v\partial_x\chi=-\partial_{vv}\mu,
\]
for some finite Radon measure $\mu\in\mathcal{M}_+.$ Hence, the entropy inequality \eqref{eq:entropyIneq1} has a kinetic formulation:
\[
\partial_t\eta_g(\rho,u)+\partial_xq_g(\rho,u)=\int_{\mathbb{R}} g(v)(\partial_t\chi+v\partial_x\chi)\di v= -\int_{\mathbb{R}} g''(\nu) \di \mu\leq0
\]
in the  distributional sense for all $g\in C_0^2(\mathbb{R})$ and $ g''\geq 0$ on the support of $\mu$.
\end{remark}

\subsubsection{Existence and uniqueness of entropy solutions of \eqref{eq:systemEuler1}}
Next, to obtain the uniqueness of entropy solutions, we consider the entropy solutions of the coupled Burgers system \eqref{eq:systemEuler1}.  We have the following proposition for entropy pairs $(\eta, q)$:
\begin{proposition}\label{pro:entropy2}
Let  $\psi_1,\psi_2\in C^2(\mathbb{R})$ be two convex functions.  Define
\begin{align}\label{eq:entroppm}
\eta(\rho,u):&=k_1\int_{\mathbb{R}}\psi'_1(v)\chi_+(v;\rho,u)\di v+k_2\int_{\mathbb{R}}\psi'_2(v)\chi_-(v;\rho,u)\di v\nonumber\\
&=k_1\psi_1(u+\sqrt{\alpha}\rho)+k_2\psi_2(u-\sqrt{\alpha}\rho),
\end{align}
and
\begin{align}\label{eq:entropfluxpm}
q(\rho,u):&=k_1\int_{\mathbb{R}}v\psi'_1(v)\chi_+(v;\rho,u)\di v+k_2\int_{\mathbb{R}}v\psi'_2(v)\chi_-(v;\rho,u)\di v\nonumber\\
&=k_1\phi_1(u+\sqrt{\alpha}\rho)+k_2\phi_2(u-\sqrt{\alpha}\rho),
\end{align}
where $k_1$ and $k_2$ are two nonnegative real numbers and $\phi$ satisfies $\phi'_i(v)=v\psi'_i(v)$ for $i=1,2$ and $v\in\mathbb{R}$.
Then, $\eta(\rho,u)$ are convex entropies with respect to $(\rho,u)$. Moreover, $q(\rho,u)$ is the corresponding entropy flux of $\eta(\rho,u)$.
\end{proposition}
\begin{proof}
The proof is similar to Proposition \ref{pro:entropy1} and we omit it.
\end{proof}
\begin{remark}\label{rmk:counterpart}
When $k_1=k_2=\frac{1}{2}$ and $\psi_1=\psi_2=\psi$, the entropy $\eta$ defined by \eqref{eq:entroppm} is equivalent to $\eta_\psi$ given in \eqref{eq:dAlembertkinetic}.
Recall Definition \ref{def:entropysolution}. For System \eqref{eq:systemEuler3}, the entropy is defined by $\eta_g$ which is one part of   \eqref{eq:dAlembertkinetic}. If we use the counter part $\eta_\psi$ in \eqref{eq:dAlembertkinetic} to define entropy class and entropy solutions of the coupled Burgers system \eqref{eq:systemEuler1}, we can also obtain global existence of solutions. This can not ensure the uniqueness of entropy solutions.  However, if we use the entropies given by \eqref{eq:entroppm}, which can be viewed as a class of entropies modifying $\eta_\psi$, to define entropy solutions of the coupled Burgers system \eqref{eq:systemEuler1}, we can obtain the stability (hence uniqueness) of solutions (see Theorem \ref{thm:wellposedness}).
\end{remark}

We give the definition of entropy solutions of the coupled Burgers system \eqref{eq:systemEuler1}.
\begin{definition}\label{def:entropysolution1}
A couple $(\rho,u)$ is called an entropy solution of the coupled Burgers system \eqref{eq:systemEuler1} if $\rho\geq0$  and it satisfies
\begin{align}\label{eq:entropyIneq}
\partial_t\eta(\rho,u)+\partial_xq(\rho,u)\leq0,
\end{align}
in distribution sense for any convex entropies $(\eta, q)$ given by \eqref{eq:entroppm}, \eqref{eq:entropfluxpm}.
\end{definition}

Next, we present an important result about the equivalent relations between entropy solutions of the coupled Burgers system \eqref{eq:systemEuler1} and solutions of the decoupled Burgers equations \eqref{eq:decouple1} and \eqref{eq:decouple2}.
\begin{proposition}\label{pro:equivalence}
If $(\rho,u)$ is an entropy solution to the coupled Burgers system \eqref{eq:systemEuler1}, then $f_\pm=u\pm\sqrt{\alpha}\rho$ are entropy solutions to the decoupled Burgers equations \eqref{eq:decouple1} and \eqref{eq:decouple2}. Conversely, if $f_\pm$ such that $f_+\geq f_-$ are entropy solutions to the decoupled Burgers equations \eqref{eq:decouple1} and \eqref{eq:decouple2}, then $(\rho,u)=\left(\frac{f_++f_-}{2},\frac{f_+-f_-}{2\sqrt{\alpha}}\right)$ is an entropy solution  to the coupled Burgers system \eqref{eq:systemEuler1}.
\end{proposition}
\begin{proof}
\textbf{Step 1.}  Assume $(\rho,u)$ is an entropy solution to the coupled Burgers system \eqref{eq:systemEuler1}. Hence, the inequality \eqref{eq:entropyIneq} holds for any $\eta$ given by \eqref{eq:entroppm}. For any convex function $\psi$, let $k_1=1$, $k_2=0$ and $\psi_1=\psi$ in \eqref{eq:entroppm}. At this time, the inequality \eqref{eq:entropyIneq} gives
\begin{align}\label{eq:entropyinede}
\partial_t\psi(f_+)+\partial_x\phi(f_+)\leq0,
\end{align}
where $\phi'(v)=v\psi'(v)$ and $f_+=u+\sqrt{\alpha}\rho$. Similarly, when $k_1=0$, $k_2=1$ and $\psi_2=\psi$, we can obtain
\begin{align}\label{eq:entropyinede1}
\partial_t\psi(f_-)+\partial_x\phi(f_-)\leq0
\end{align}
in distribution sense.
 Inequalities \eqref{eq:entropyinede} and  \eqref{eq:entropyinede1} are exactly the entropy inequalities for the decoupled Burgers equations \eqref{eq:decouple1} and \eqref{eq:decouple2}. Hence, $f_\pm$ are entropy solutions to \eqref{eq:decouple1} and  \eqref{eq:decouple2}.

\textbf{Step 2.} Let $f_\pm$ be an entropy solution of the decoupled Burgers equations \eqref{eq:decouple1} and \eqref{eq:decouple2}. Due to $f_+\geq f_-$, we have $\rho=\frac{f_+-f_-}{2\sqrt{\alpha}}\geq0$. Moreover, inequality \eqref{eq:entropyinede} holds for any entropy pair $(\psi_1,\phi_1)$ with $\phi'_1(v)=v\psi'_1(v)$, and inequality \eqref{eq:entropyinede1} holds for any entropy pair $(\psi_2,\phi_2)$ with $\phi'_2(v)=v\psi'_2(v)$. The linear combination of \eqref{eq:entropyinede} and \eqref{eq:entropyinede1} with nonnegative coefficients $k_1$ and $k_2$ generates the inequality \eqref{eq:entropyIneq}.
Hence, $(\rho,u)$ is an entropy solution to the coupled Burgers system \eqref{eq:systemEuler1}.

\end{proof}

Due to the well-posedness of the scalar conservation law (Burgers equation), we have the following well-posedness result for the coupled Burgers system \eqref{eq:systemEuler1}:
\begin{theorem}\label{thm:wellposedness}
Let $\rho_0(x)$ and $u_0(x)$ be two bounded measurable functions satisfying $\rho_0\geq0$. Then:

$\mathrm{(i)}$ There exist a unique entropy solution $(\rho(x,t),u(x,t))$ to the coupled Burgers system \eqref{eq:systemEuler1} such that $\rho\geq0$ and $(\rho,u)|_{t=0}=(\rho_0,u_0)$.

$\mathrm{(ii)}$ Let $(\tilde{\rho},\tilde{u})$ be another entropy solution of the coupled Burgers system \eqref{eq:systemEuler1} subject to initial datum $(\tilde{\rho}_0(x),\tilde{u}_0(x))$  with $\tilde{\rho}_0\geq0$. If $u_0-\tilde{u}_0,~\rho_0-\tilde{\rho}_0\in L^1(\mathbb{R})$,  then $u(\cdot,t)-\tilde{u}(\cdot,t),~\rho(\cdot,t)-\tilde{\rho}(\cdot,t)\in L^1(\mathbb{R})$ and
\begin{align}\label{eq:stability}
\sqrt{\alpha}\|\rho(\cdot,t)-\tilde{\rho}(\cdot,t)\|_{L^1}+\|u(\cdot,t)-\tilde{u}(\cdot,t)\|_{L^1}\leq 2(\sqrt{\alpha}\|\rho_0-\tilde{\rho}_0\|_{L^1}+\|u_0-\tilde{u}_0\|_{L^1}).
\end{align}
\end{theorem}

\begin{proof}
(i) Consider the decoupled Burgers equations \eqref{eq:decouple1} and \eqref{eq:decouple2} with initial datum $f_\pm(x,0):=u_0(x)\pm\sqrt{\alpha}\rho_0(x)$. Then, there is a unique entropy solutions $f_\pm(x,t)$ to \eqref{eq:decouple1} and \eqref{eq:decouple2} respectively. Due to $\rho_0\geq0$, we have $f_+(x,0)\geq f_-(x,0)$. Hence, from \cite[Proposition 2.3.6 ]{Serre1999Systems}, we have $f_+(x,t)\geq f_-(x,t)$ for any $t>0$ and $x\in\mathbb{R}$. By Proposition \ref{pro:equivalence}, there is a unique solution to the coupled Burgers system \eqref{eq:systemEuler1} given by
\begin{align}
u(x,t)=\frac{f_+(x,t)+f_-(x,t)}{2},\quad \rho(x,t)=\frac{f_+(x,t)-f_-(x,t)}{2\sqrt{\alpha}}.
\end{align}
Moreover, we have $\rho\geq0$.

(ii) Let $f_\pm:=u\pm\sqrt{\alpha}\rho$ and $\tilde{f}_\pm:=\tilde{u}\pm\sqrt{\alpha}\tilde{\rho}$. Then, $f_\pm$ and $\tilde{f}_\pm$ are entropy solutions to the decoupled Burgers equations \eqref{eq:decouple1} and \eqref{eq:decouple2}. By the stability results for scalar conservation law (see \cite[Proposition 2.3.6 ]{Serre1999Systems}), we have
\begin{align*}
&\sqrt{\alpha}\|\rho(\cdot,t)-\tilde{\rho}(\cdot,t)\|_{L^1}+\|u(\cdot,t)-\tilde{u}(\cdot,t)\|_{L^1}\nonumber\\
=&\left\|\frac{f_{+}(\cdot,t)-f_{-}(\cdot,t)}{2}-\frac{\tilde{f}_{+}(\cdot,t)-\tilde{f}_{-}(\cdot,t)}{2}\right\|_{L^1}\nonumber\\
&\qquad\qquad\qquad\qquad\qquad\qquad+\left\|\frac{f_{+}(\cdot,t)+f_{-}(\cdot,t)}{2}-\frac{\tilde{f}_{+}(\cdot,t)+\tilde{f}_{-}(\cdot,t)}{2}\right\|_{L^1}\nonumber\\
\leq& \|f_{+}(\cdot,0)-\tilde{f}_{+}(\cdot,0)\|_{L^1}+\|f_{-}(\cdot,0)-\tilde{f}_{-}(\cdot,0)\|_{L^1} \nonumber\\
\leq &2(\sqrt{\alpha}\|\rho_0-\tilde{\rho}_0\|_{L^1}+\|u_0-\tilde{u}_0)\|_{L^1}).
\end{align*}
\end{proof}
\begin{remark}
We remark that $f_+$ and $f_-$ are entropy solutions to the decoupled Burgers equations \eqref{eq:decouple1} and \eqref{eq:decouple2} respectively if and only if there are two positive Radon measures $\mu_+,\mu_-\in\mathcal{M}_{+}(\mathbb{R})$ such that the kinetic functions $\chi_\pm(v;\rho,u)$ given by \eqref{eq:kineticf} are  weak solution of the kinetic equations \cite{Perthame2000Kinetic}
\[
\partial_t\chi_\pm+v\partial_x\chi_\pm=\partial_{v}\mu_\pm.
\]
Actually, for an entropy pair $(\psi,\phi)$, one has
\[
\partial_t\psi(f_\pm)+\partial_x\phi(f_\pm)=\int_\mathbb{R}\psi'(v)(\partial_t\chi_\pm+v\partial_x\chi_\pm)\di v=-\int_\mathbb{R}\psi''(v)\mu_\pm \di v\leq0
\]
in distribution sense.
{For more detailed discussions of using these kinetic density functions to study the coupled Burgers system \eqref{eq:systemEuler1} for $(\rho, u)$, one can refer to \cite{Perthame2000Kinetic}.}
\end{remark}

\begin{remark}[Difference between Definition \ref{def:entropysolution} and Definition \ref{def:entropysolution1}]
From \eqref{eq:Hamiltonianu12}, the Hamiltonian $H_1^u$ corresponds to \eqref{eq:entroppm} for $\psi_1(v)=\psi_2(v)=v^3$ and $k_1=k_2=1/12$.  At this time $\psi$ is not convex and hence from Proposition \ref{pro:entropy1}, we know that $H_1^u$ is not convex with respect to $(\rho,u)$. When $\rho\geq0$, we also have $H_2^u$ is convex with respect to $(\rho,u)$. However, we have
\[
H_2^u(\rho,u)=\frac{1}{4}\int_{\mathbb{R}}v^2\chi_+(v;\rho,u)\di v-\frac{1}{4}\int_{\mathbb{R}}v^2\chi_-(v;\rho,u)\di v,
\]
which is not a proper entropy as in Definition \ref{def:entropysolution1}.

Similarly, one can show that $H_1^m$  is not a proper entropy as in Definition \ref{def:entropysolution}, while $H_2^m$ is a convex entropy for system \eqref{eq:systemEuler3}.

\end{remark}

To end this subsection, we give the kinetic formulation for the well known Lax entropy \cite{lax1957hyperbolic}. Let the solution of the wave equation  \eqref{eq:waveequation} have the form $\eta(k;\rho, u)=e^{ku}\sigma(k;\rho)$ for some constant parameter $k\neq0$. Then, equation  \eqref{eq:waveequation} becomes the ODE
\[
\sigma_{\rho\rho}(\rho)=\alpha k^2\sigma(\rho),\quad \sigma(0)=0,~~\sigma'(0)=1.
\]
Hence, we have
\[
\sigma(\rho)=\frac{e^{\sqrt{\alpha}k\rho}-e^{-\sqrt{\alpha}k\rho}}{2\sqrt{\alpha}k},\quad \eta(k;\rho,u)=\frac{e^{kf_+}-e^{kf_-}}{2\sqrt{\alpha}k}.
\]
This yields a family of Lax entropy pairs:
\begin{align}\label{eq:Laxentropypair}
\eta(k;\rho, u) = \frac{e^{k f_+}-e^{kf_-}}{2\sqrt{\alpha}k},\quad q(k;\rho, u)  = \frac{(kf_+-1)e^{kf_+}-(kf_--1)e^{kf_-}}{2\sqrt{\alpha}k^2}.
\end{align}
Note that both $\eta$ and $q$ are real functions. When $g(v)=e^{kv}$ in \eqref{eq:entropyg}, $\eta_g(\rho,m)$ recovers the Lax entropy given in \eqref{eq:Laxentropypair}.

\section{Lagrangian dynamics for \eqref{eq:systemEuler1} and its relation with Calogero-Moser model}\label{sec:Lagrangemass}
In this section, we derive the Lagrangian dynamics for the coupled Burgers system \eqref{eq:systemEuler1}, which recovers the dynamics \eqref{eq:gas} for gas. Moreover, we present a nonlinear spring-mass system (Fermi-Pasta-Ulam-Tsingou model) with nearest-neighbor interactions and its continuum limit yields the Lagrangian dynamics of the coupled Burgers system \eqref{eq:systemEuler1}.

\subsection{Lagrangian dynamics for  the coupled Burgers system \eqref{eq:systemEuler1}}\label{sec:Lagrange}
Consider an initial datum for the coupled Burgers system \eqref{eq:systemEuler1}:
\begin{align}\label{eq:systemEuler1initial}
u(x,0)=u_0(x),\quad \rho(x,0)=\rho_0(x),~~x\in\mathbb{R}.
\end{align}
Assume that initial density function $\rho_0:\mathbb{R}\rightarrow\mathbb{R}$ satisfies $\rho_0(x)>0$ and the  total mass $||\rho_0||_{L^1}=1$. Define the initial cumulative mass distribution function $Z_0$:
\begin{align}\label{eq:CumulativeMD}
Z_0(x):=\int_{-\infty}^x\rho_0(y)\di y~\textrm{ for }~x\in\mathbb{R}.
\end{align}
Then, function $Z_0:\mathbb{R}\to(0,1)$ is strictly increasing.
Hence, there is an inverse function $X_0:(0,1)\to\mathbb{R}$ such that
\begin{align}\label{eq:definitial}
Z_0(X_0(\xi))=\xi, ~~X_0(Z_0(x))=x~\textrm{ for }~x\in \mathbb{R}, ~~\xi\in(0,1).
\end{align}
Moreover, we have
\begin{align}\label{eq:initialheight}
Z_0(0)=X_0(0)=0,~\textrm{ and }~\frac{1}{X'_0(\xi)}=Z'_0(x)=\rho_0(x)~\textrm{ for }~\xi=Z_0(x).
\end{align}
Here, $x$ is the Eulerian coordinates and we take $\xi$ as the Lagrangian coordinates.

Give an Eulerian velocity field $u:\mathbb{R}\times[0,\infty)\rightarrow\mathbb{R}$. Define the flow map $X(\xi,t)$ satisfying
\begin{gather}\label{eq:leastdynamics}
\left\{
\begin{split}
&\dot{X}(\xi,t)=u(X(\xi,t),t), ~~\xi\in (0,1),~t>0,\\
&X(\xi,0)=X_0(\xi).
\end{split}
\right.
\end{gather}
Here, $\dot{X}(\xi,t)$ denotes $\partial_tX(\xi,t)$.
Hence, we have $\partial_\xi \dot{X} = \partial_x u \partial_\xi X$ and thus
\begin{align}\label{eq:Xxixi}
\partial_\xi X(\xi,t)=X'_0(\xi)e^{\int_0^t\partial_xu(X(\xi,s),s)\di s}>0,~~\xi\in(0,1).
\end{align}
Define the density function in Lagrangian coordinates at time $t$ as:
\begin{align}\label{eq:densityfunction}
\rho(X(\xi,t),t):=\frac{1}{\partial_\xi X(\xi,t)},
\end{align}
Hence,
\begin{align}\label{eq:densitymap}
\rho(x,t)\di x=\di\xi~\textrm{ and }~\partial_t\rho+\partial_x(\rho u)=0, \quad \rho(x,0)=\rho_0(x),
\end{align}
which is the first equation in the coupled Burgers system \eqref{eq:systemEuler1}.
We also have local mass conservation law:
\[
\int_{X(\xi_1,t)}^{X(\xi_2,t)} \rho(x,t) \di x =\xi_1-\xi_2= \int^{X_0(\xi_1)}_{X_0(\xi_2)} \rho_0(x) \di x~\textrm{ for any }~\xi_i\in (0,1),~i=1,2.
\]
By \eqref{eq:densityfunction}, we obtain
\[
\partial_x\rho(X(\xi,t),t)=\frac{1}{\partial_\xi X(\xi,t)}\partial_\xi\left(\frac{1}{\partial_\xi X(\xi,t)}\right)=-\frac{\partial_{\xi\xi}X(\xi,t)}{(\partial_\xi X)^3(\xi,t)},
\]
which gives
\begin{align}\label{eq:Lagrangerhorhox}
-(\alpha\rho\partial_x\rho)(X(\xi,t),t)=\alpha\frac{\partial_{\xi\xi}X(\xi,t)}{(\partial_\xi X)^4(\xi,t)}.
\end{align}
Set
\begin{align}\label{eq:DefV}
V(\xi,t):=u(X(\xi,t),t).
\end{align}
Combining \eqref{eq:Lagrangerhorhox},  the coupled Burgers system  \eqref{eq:systemEuler1}  is recast to the Lagrangian dynamics:
\begin{gather}\label{eq:LagrangeEuler}
\left\{
\begin{split}
&\dot{X}(\xi,t)=V(\xi,t), \quad \xi\in (0,1),~t>0,\\
&\dot{V}(\xi,t)=\alpha\frac{\partial_{\xi\xi}X(\xi,t)}{(\partial_\xi X)^4(\xi,t)}=-\frac{\alpha}{3}\partial_\xi\left(\frac{1}{(\partial_\xi X)^3(\xi,t)}\right),
\end{split}
\right.
\end{gather}
subject to initial datum
\begin{gather}\label{eq:LagrangeEulerinitial}
\left\{
\begin{split}
&X(\xi,0)=X_0(\xi), \quad \xi\in (0,1),\\
&V(\xi,0)=u_0(\xi).
\end{split}
\right.
\end{gather}
Here, $u_0$ is given by \eqref{eq:systemEuler1initial} and $X_0(\xi)$ is given by \eqref{eq:definitial}.
Taking derivative of the first equation in \eqref{eq:LagrangeEuler} with respect to $\xi$, we can recover the dynamics  \eqref{eq:gas} for gas with $\tau(\xi,t):=X_\xi(\xi,t)$.

Next, we briefly show least action principle for the Lagrangian dynamics \eqref{eq:LagrangeEuler}.
Corresponding to the total energy $H_2^m(\rho,m)$ given by \eqref{eq:mrho}, we use Legendre transformation to obtain the Lagrangian functional as
\begin{align*}
\mathscr{L}(\rho,u)=\int_{\mathbb{R}}m \frac{\delta H_2^m}{\delta m}  \di x-H_2^m(\rho,m)=\int_{\mathbb{R}} \left(\frac{1}{2}\rho u^2-  \frac{\alpha}{6}\rho^3\right)\di x
\end{align*}	
The momentum $m$ is recovered by taking the variation of $\mathscr{L}$ with respect to $u$:
\[
m = \frac{\delta \mathscr{L}}{\delta u}=\rho u.
\]
The action  is defined by
\begin{align}\label{eq:actionforEuler}
\mathcal{A}(X)=\frac{1}{2}\int_0^1\int_{\mathbb{R}} \left(\rho u^2 -  \frac{\alpha}{3}\rho^3\right)\di x\di t=\frac{1}{2}\int_0^1\int_{0}^1  \left(\dot{X}^2(\xi,t) - \frac{\alpha}{3(\partial_\xi X)^2(\xi,t)}\right) \di \xi \di t.
\end{align}
Next, consider two increasing functions for $\xi\in [0,1]$: $X(\xi,0)=X_0(\xi)$ and $X(\xi,1)=X_1(\xi)$. We formally show that the coupled Burgers system \eqref{eq:systemEuler1} corresponds to a critical path of the action $\mathcal{A}(X)$ in some manifold connecting $X_0$ and $X_1$ for $t\in[0,1]$.
For any $Y\in C_c^\infty((0,1)\times(0,1))$, we have
\begin{multline*}
\int_0^1\int_{0}^1\frac{\delta \mathcal{A}}{\delta X}\cdot Y\di \xi \di t=\lim_{\epsilon\rightarrow0}\frac{\mathcal{A}(X+\epsilon Y)-\mathcal{A}(X)}{\epsilon}\\
=\frac{1}{2}\int_0^1\int_{0}^1\left(2\dot{X}\dot{Y} + \frac{2\alpha}{3(\partial_\xi X)^3}\partial_\xi Y\right) \di \xi \di t=\int_0^1\int_{0}^1\left[-\ddot{X} - \partial_\xi\left(\frac{\alpha}{3(\partial_\xi X)^3}\right)\right]Y \di \xi \di t.
\end{multline*}
This gives 
\begin{align*}
\frac{\delta \mathcal{A}}{\delta X}=-\ddot{X}-\partial_\xi \left(\frac{\alpha}{3(\partial_\xi X)^3}\right).
\end{align*}
Take $\displaystyle{\frac{\delta A}{\delta X}=0},$ and we have
\begin{align}\label{eq:action}
\ddot{X}-\alpha\frac{\partial_{\xi\xi} X}{(\partial_\xi X)^4}=0,
\end{align}
which corresponds to the Lagrangian dynamics \eqref{eq:LagrangeEuler}.

\subsection{A spring-mass system with nearest-neighbor interactions}
In this subsection, we present a local interaction model for $N$ masses and show that the Lagrangian dynamics system \eqref{eq:LagrangeEuler} is exactly the continuum limit equation of this model. For $N$ ordered masses $x_1(t)<\cdots<x_N(t)$, each mass is evolved by a force generated by interactions between nearest neighbors and the model is described by
\begin{gather}\label{eq:springModel}
\left\{
\begin{split}
&\dot{x}_j(t)=v_j(t) , \quad 1\leq j\leq N,\\
&\dot{v}_j(t)=\frac{\alpha}{3N^2}\left[\frac{1}{(x_{j+1}(t)-x_j(t))^3}+\frac{1}{(x_{j-1}(t)-x_{j}(t))^3}\right].
\end{split}
\right.
\end{gather}
Here we assume
\begin{align}
x_0=x_{N+1}=+\infty,~\textrm{ and }~\frac{1}{(x_{0}(t)-x_1(t))^3}=\frac{1}{(x_{N+1}(t)-x_N(t))^3}=0.
\end{align}
{The masses accelerated by an repulsive force if  $\alpha<0$. While $\alpha>0$, the masses attract each other.}
System \eqref{eq:springModel} is a Hamiltonian system corresponding to the Hamiltonian functional:
\begin{align}\label{eq:CMHamiltonian1}
{H}(x,p)=\frac{N}{2}\sum_{j=1}^Np_j^2-\frac{\alpha}{12N^3}\sum_{j=1}^N\sum_{k= j\pm 1}\frac{1}{(x_j-x_{k})^2}.
\end{align}
Momentum $p_i$ equals to mass $1/N$ times velocity $v_i$ which means $v_i=Np_i$. Hence, \eqref{eq:springModel} equals to
\begin{gather*}
\left\{
\begin{split}
&\dot{x}_j(t)=\partial_{p_j}{H},\quad \quad 1\leq j\leq N,\\
&\dot{p}_j(t)=-\partial_{x_j}{H}.
\end{split}
\right.
\end{gather*}

Model \eqref{eq:springModel} describes local interactions between masses and their nearest-neighbors, which is a special case of the Fermi-Pasta-Ulam-Tsingou lattice system. We compare \eqref{eq:springModel} with another Fermi-Pasta-Ulam-Tsingou lattice system, Toda lattice,  given by the system of ordinary differential equations
\begin{align}\label{eq:Toda}
\frac{\di^2q_j}{\di t^2}=e^{q_{j+1}-q_j}-e^{q_{j}-q_{j-1}},~~j\in\mathbb{Z}.
\end{align}
Note that Toda lattice is an integrable system. We do not know whether System \eqref{eq:springModel} is an integrable system or not. However, if each mass interacts with all the other masses with the same manner, we can obtain an integrable global interaction model, the Calogero-Moser model (see Remark \ref{rmk:Calogero-Moser model}).

Next, we formally derive the continuum limit of the local interaction mass system. To do this, we assume the masses initially distribute uniformly and $x_j(t)=X(\xi,t)$, $x_{j+1}(t)=X(\xi+1/N,t)$ and $x_{j-1}(t)=X(\xi-1/N,t)$ for some $\xi\in(0,1)$ and $2\leq j\leq N-1$. Using Taylor expansion,  we have
\[
x_{j+1}(t)-x_j(t)=\partial_{\xi}X(\xi,t)N^{-1}+\frac{1}{2}\partial_{\xi\xi}X(\xi,t)N^{-2}+\frac{1}{6}\partial_{\xi\xi\xi}X(\xi,t)N^{-3}+O((N^{-4}),
\]
and
\[
x_{j-1}(t)-x_j(t)=-\partial_{\xi}X(\xi,t)N^{-1}+\frac{1}{2}\partial_{\xi\xi}X(\xi,t)N^{-2}-\frac{1}{6}\partial_{\xi\xi\xi}X(\xi,t)N^{-3}+O((N^{-4}).
\]
Hence, we can obtain
\begin{align*}
&\frac{\alpha}{3N^2}\left[\frac{1}{(x_{j+1}(t)-x_{j}(t))^3}+\frac{1}{(x_{j-1}(t)-x_{j}(t))^3}\right]\nonumber\\
=&\frac{\alpha}{3N^2}\frac{(x_{j-1}+x_{j+1}-2x_j)\Big[(x_{j-1}-x_{j})^2+(x_{j+1}-x_{j})^2-(x_{j-1}-x_{j})(x_{j+1}-x_{j})\Big]}{(x_{j+1}-x_{j})^3(x_{j-1}-x_{j})^3}\nonumber\\
=&\frac{\alpha}{3N^2}\frac{\Big[\partial_{\xi\xi}X (\xi,t)N^{-2}+O(N^{-4})\Big]\cdot\Big[3(\partial_{\xi}X)^2(\xi,t)N^{-2}+O(N^{-3})\Big]}{(\partial_{\xi}X)^6(\xi,t)N^{-6}+O(N^{-7})}\nonumber\\
=&\frac{\alpha}{3N^2}\frac{N^{-4}}{N^{-6}}\frac{\Big[\partial_{\xi\xi}X(\xi,t)+O(N^{-2})\Big]\cdot\Big[3(\partial_{\xi}X)^2(\xi,t)+O(N^{-1})\Big]}{(\partial_{\xi}X)^6(\xi,t)+O(N^{-1})}\nonumber\\
=&\alpha\frac{\partial_{\xi\xi}X(\xi,t)+O(N^{-1})}{(\partial_{\xi}X)^4(\xi,t)+O(N^{-1})}.
\end{align*}
Let $N\to \infty$ and we obtain
\[
\lim_{N\to\infty}\frac{\alpha}{3N^2}\left[\frac{1}{(x_{j+1}(t)-x_{j}(t))^3}+\frac{1}{(x_{j-1}(t)-x_{j}(t))^3}\right]=\alpha\frac{\partial_{\xi\xi}X(\xi,t)}{(\partial_{\xi}X)^4(\xi,t)}.
\]
This gives the continuum coupled Burgers system in Lagrangian coordinate \eqref{eq:LagrangeEuler}.

\begin{remark}\label{rmk:Calogero-Moser model}
If each mass interact with all the other masses with the same manner (the force between each pair of two masses are  reciprocal proportion to the cubic of distance between them), we can obtain an integrable global interaction model, the Calogero-Moser model \cite{moser1976three}:
\begin{gather}\label{eq:CMmodel}
\left\{
\begin{split}
&\dot{x}_j(t)=v_j(t) ,\\
&\dot{v}_j(t)=\frac{4\alpha}{N^2\pi^2}\sum_{k=1,k\neq j}^N\frac{1}{(x_j(t)-x_k(t))^3},\quad 1\leq j\leq N.
\end{split}
\right.
\end{gather}
The coefficients of \eqref{eq:CMmodel} are different from the coefficients in \eqref{eq:springModel}. System \eqref{eq:CMmodel} is also a Hamiltonian system and the rescaled ($p_j=v_j/N$) Hamiltonian is given by
\begin{align}\label{eq:CMHamiltonian}
\tilde{H}(x,q)=\frac{N}{2}\sum_{j=1}^Np_j^2+\frac{\alpha}{2N^3\pi^2}\sum_{j=1}^N\sum_{k\neq j}\frac{1}{(x_j-x_k)^2}.
\end{align}
By using the Euler-MacLaurin asymptotic expansion for the Riemann integral of functions, Menon  \cite{Menon2} showed that System \eqref{eq:LagrangeEuler} is the $N\to \infty$ limit of the Calogero-Moser system corresponding to the rescaled Hamiltonian \eqref{eq:CMHamiltonian}.
As shown by \cite[Eqs. (5.13),(5.26)]{Menon2}, the Hamiltonian $\tilde{H}$ corresponds to the total energy $H_2^m$ (see \eqref{eq:Hamilton1}) of System \eqref{eq:systemEuler3}.

\end{remark}

\section*{Acknowledgements}
We are grateful to Govind Menon for some helpful discussions. We would like to thank the support by the National Science Foundation under grants DMS 1514826 and 1812573 (JGL).

\bibliographystyle{plain}
\bibliography{bibofCH}

\appendix

\section{Proof of Theorem \ref{thm:analytic}}\label{App_B}
Consider the initial datum given by \eqref{eq:initialf}.
Direct calculation shows that
\begin{align*}
f_0(z)=\frac{1}{\pi}\int_{\mathbb{R}}\frac{\rho_0(s)}{z-s}\di s&=\frac{1}{\pi}\int_{\mathbb{R}}\frac{x-s}{y^2+(x-s)^2}\rho_0(s)\di s-\frac{i}{\pi}\int_{\mathbb{R}}\frac{y}{y^2+(x-s)^2}\rho_0(s)\di s\\
&=:R\rho_0(x,y)-iP\rho_0(x,y),
\end{align*}
where $P\rho_0(x,y)$ and $R\rho_0(x,y)$ are given by the convolution of $\rho_0$ with the Poisson kernel and the conjugate Poisson kernel given by:
\begin{align}\label{eq:PoissonKernel}
P_y(x):=\frac{1}{\pi}\frac{y}{y^2+x^2}~\textrm{ and }~R_y(x):=\frac{1}{\pi}\frac{x}{y^2+x^2}.
\end{align}
Furthermore, we have
\[
\lim_{y\to0+}[R\rho_0(x,y)-iP\rho_0(x,y)]=H\rho_0(x)-i\rho_0(x)~\textrm{for a.e.}~x\in\mathbb{R}.
\]
Recall that the following properties of Poisson kernel:
\begin{enumerate}
\item[(i)] If $h\in L^2(\mathbb{R})$,  then
\[
Rh(x,y)=PHh(x,y)~\textrm{ on }~\mathbb{R}^2_+.
\]
\item[(ii)] If $h\in L^\infty(\mathbb{R})$ and is vanishing at infinity, then
\[
\lim_{y\to+\infty}Ph(x,y)=0,~~x\in\mathbb{R},
\]
and
\[
\lim_{x\to\pm\infty}Ph(x,y)=0,~~y\geq0.
\]
\item[(iii)] If $h\in L^\infty(\mathbb{R})$, then $Ph(x,y)$ is a bounded function on $\mathbb{R}^2_+$.
\end{enumerate}

Next, we prove the existence and uniqueness of $\mathbb{C}_+$-holomorphic solutions to  \eqref{eq:complexBurgers2} by the characteristics method.
Consider the characteristics given by
\begin{equation}\label{eq:complexBurgerscharac}
\frac{\di}{\di t}Z(w,t)=g(Z(w,t),t),\quad Z(w,0)=w\in \mathbb{C}_+.
\end{equation}
Then,
\[
\frac{\di^2}{\di t^2 } Z(w,t)=\frac{\di}{\di t}g(Z(w,t),t)=[\partial_t g+g\partial_zg](Z(w,t),t)=\gamma^2Z(w,t),
\]
with initial date
\[
Z(w,0)=w,\quad \frac{\di}{\di t}Z(w,t)\Big|_{t=0}=g_0(w),~~w\in\mathbb{C_+}.
\]
Equation \eqref{eq:complexBurgerscharac} gives the following complex trajectories:
\begin{gather}\label{eq:gloabltrajec}
Z(w,t)=\left\{
\begin{split}
&w \cosh \gamma t +\frac{1}{\gamma}g_0(w)\sinh \gamma t,~~\gamma>0,\\ 
&g_0(w)t+w=f_0(w)t+w,~~\gamma=0.
\end{split}
\right.
\end{gather}
Here, we only treat the case for $\gamma>0$ where the convergence to the steady state for analytical solutions happens. For the well-posedness results of the case $\gamma=0$, one can refer to \cite{Castro2008Global}.  Let 
\[
Z(w,t)=Z_1(x,y,t)+iZ_2(x,y,t),~~w=x+iy\in \mathbb{C_+},
\]
and we have
\begin{align}
Z_1(x,y,t)=x \cosh  \gamma  t+ \frac{\pi}{\gamma} R\rho_0(x,y)\sinh t-x\sinh \gamma t=x e^{-\gamma t}+\frac{ \pi}{\gamma} R\rho_0(x,y)\sinh \gamma t,\label{eq:realpart}\\
Z_2(x,y,t)=y \cosh \gamma  t- \frac{\pi}{\gamma} P\rho_0(x,y)\sinh \gamma t-y\sinh \gamma t= ye^{-\gamma t}-\frac{ \pi}{\gamma} P\rho_0(x,y)\sinh \gamma t.\label{eq:imaginarypart}
\end{align}
Because the initial date $g_0(w)$ in \eqref{eq:complexBurgers2} is a  $\mathbb{C}_+$-holomorphic function, $Z(w,t)$ given by \eqref{eq:gloabltrajec} is $\mathbb{C}_+$-holomorphic of $w$ for any $t\geq0$. Next, we present a lemma to show that for any fixed time $t>0$ the backward characteristics of \eqref{eq:gloabltrajec} are well defined on the set $\occ$. We have:
\begin{lemma}\label{lmm:bijection}
Let $0<\rho_0\in  H^s(\mathbb{R})\cap L^1(\mathbb{R})$ with $s>1/2$. For fixed  $t_0> 0$ and fixed $Z=Z_1+iZ_2\in \occ$, there exists a unique $w=x+iy\in \mathbb{C}_+$ such that \eqref{eq:realpart} and \eqref{eq:imaginarypart} hold. 

\end{lemma}
\begin{proof}
Given $t_0>0$, denote
\[
a:=e^{- \gamma t_0},\quad b:=\frac{\pi}{\gamma}\sinh \gamma  t_0.
\]
Then \eqref{eq:realpart} and \eqref{eq:imaginarypart} become
\begin{align*}
Z_1 = ax+ b R\rho_0(x,y), \quad Z_2= ay-b P\rho_0(x,y).
\end{align*}

\textbf{Step 1.} In this step, we prove that for any $x$, there exists a unique $y>0$ satisfies \eqref{eq:imaginarypart} for $Z_2\geq0$ and $t_0>0$.

Because $P\rho_0(x,y)>0$ is a bounded function on $\mathbb{R}^2_+$,  by the property of Poisson kernel we have
\begin{align}\label{eq:proper1}
\lim_{y\to+\infty}Z_2(x,y,t_0)=+\infty,~~\lim_{y\to0+}Z_2(x,y,t_0)=-b \rho_0(x)<0.
\end{align}
Hence, for any fixed $Z_2\geq0$, there exists a point $y>0$ depending on $x$ such that
\[
Z_2=ay - bP\rho_0(x,y).
\]
Now we prove that  $y$ is unique. Suppose that there exist $y_1>y_2$ such that
\begin{align*}
Z_2=ay_1-b P\rho_0(x,y_1),\\
Z_2=ay_2-b P\rho_0(x,y_2),
\end{align*}
which implies
\[
y_1,y_2>Z_2/a~\textrm{ and }~\frac{P\rho_0(x,y_1)}{y_1-Z_2/a}=\frac{P\rho_0(x,y_2)}{y_2-Z_2/a}.
\]
Because function
\[
h(y)=\frac{y}{y-Z_2/a}\cdot \frac{1}{y^2+(x-s)^2}
\]
is a decreasing function for $y>Z_2/a$, we obtain a contradiction.

Now we denote by $y_{Z_2}(x)$ the solution of \eqref{eq:imaginarypart} with fixed $Z_2\geq0$, $t_0>0$ and $x\in\mathbb{R}$. Hence, we obtain
\begin{align}\label{eq:yx}
ay_{Z_2}(x)-Z_2=bP\rho_0(x,y_{Z_2}(x)).
\end{align}

\textbf{Step 2.} In this step, we prove there exits a unique $x$ satisfies \eqref{eq:realpart} for fixed $Z_1,Z_2$ and $t_0.$
Taking derivative of \eqref{eq:yx} with respect to $x$ gives
\begin{align}\label{eq:derivativeyx}
\frac{\di }{\di x}y_{Z_2}(x)= \frac{\partial_xP\rho_0(x,y_{Z_2}(x))}{a/b-\partial_yP\rho_0(x,y_{Z_2}(x))}.
\end{align}
Since $\rho_0\in H^s(\mathbb{R})\cap L^1(\mathbb{R})$ ($s>1/2$), it follows that $H\rho_0\in L^\infty(\mathbb{R})$ and therefore $R\rho_0=PH\rho_0$ is a bounded function over $\mathbb{R}^2_+$. Furthermore,
\begin{align}\label{eq:proper2}
\lim_{x\to\pm\infty}[ax+bR\rho_0(x,y_{Z_2}(x))]=\pm\infty.
\end{align}
Hence, for any $Z_1\in\mathbb{R}$, we can find a $x\in\mathbb{R}$ such that
\[
Z_1=ax+bR\rho_0(x,y_{Z_2}(x)).
\]
To prove the uniqueness, we only have to prove the following function
\[
q(x)=ax+b R\rho_0(x,y_{Z_2}(x)),
\]
is an increasing function. By using \eqref{eq:derivativeyx} and the Cauchy-Riemann equations
\begin{align}\label{eq:CauchyRiemann}
\partial_x R\rho_0=-\partial_yP\rho_0,\quad \partial_x P\rho_0=\partial_yR\rho_0,
\end{align}
and taking derivative of $q(x)$ gives
\[
\frac{\di}{\di x}q(x)=\frac{b(a/b+\partial_x R\rho_0)^2+(\partial_x P\rho_0)^2}{a/b+\partial_x R\rho_0}(x,y_{Z_2}(x)).
\]
To prove the increasing of $q(x)$, it is sufficient to prove
\begin{align}\label{eq:increasing}
a/b +\partial_x R\rho_0(x,y)>0
\end{align}
for any $(x,y)\in\mathbb{R}^2_+$ satisfying $ay-bP\rho_0(x,y)\geq0$ and $y>0$.
Suppose that
\[
a/b+\partial_x R\rho_0(x_0,y_0)\leq 0
\]
for some point $(x_0,y_0)\in\mathbb{R}^2_+$ with $ay_0-bP\rho_0(x_0,y_0)\geq0$. Then, we have
\begin{align*}
-a/b\geq \partial_x R\rho_0(x_0,y_0)&=\frac{1}{\pi}\int_{\mathbb{R}}\frac{y_0^2-(x_0-s)^2}{[y_0^2+(x_0-s)^2]^2}\rho_0(s)\di s>\frac{1}{\pi}\int_{\mathbb{R}}\frac{-y_0^2-(x_0-s)^2}{[y_0^2+(x_0-s)^2]^2}\rho_0(s)\di s\nonumber\\
&=\frac{1}{\pi}\int_{\mathbb{R}}\frac{-1}{y_0^2+(x_0-s)^2}\rho_0(s)\di s=-\frac{P\rho_0(x_0,y_0)}{y_0},
\end{align*}
which implies a contradiction:
\[
ay_0-bP\rho_0(x_0,y_0)<0.
\]
\end{proof}
From the above lemma, we know that the backward characteristics are well defined. More importantly, for any $Z\in \occ$ the initial point $w$ must be interior point in $\mathbb{C}_+$.  For any $t\geq0$, we denote the backward characteristics as:
\[
Z^{-1}(\cdot,t): \occ \to \mathbb{C}_+.
\]
From the uniqueness in Lemma \ref{lmm:bijection}, $Z^{-1}(\cdot,t)$ is an $1-1$ map.

\begin{proof}[Proof of Theorem \ref{thm:analytic}]
For simplicity, we only consider the case $\gamma=1$. The proof for arbitrary $\gamma>0$ is similar.

\textbf{Step 1.} Proof of (i).
From Lemma \ref{lmm:bijection}, we have $\occ\subset\{Z(w,t):~~w\in \mathbb{C}_+\}$ and $Z^{-1}(\cdot,t)$ is well defined on $\occ$ for any fixed time $t>0$. Denote the preimage of $Z(\cdot,t)$ as:
\[
Z^{-1}(\occ,t):=\{w\in \mathbb{C}_+;~~Z(w,t)\in \occ\}.
\]
Denote
\[
a(t):=e^{-  t},\quad b(t):={\pi}\sinh   t.
\]
For $(x,y)\in\mathbb{R}^2_+$ and $Z_2(x,y,t)\geq0$, by the Cauchy-Riemann equation \eqref{eq:CauchyRiemann},  we have
\begin{align}\label{eq:big0}
|Z_w(w,t)|=\left|\frac{\partial(Z_1,Z_2)}{\partial (x,y)}\right|(x,y)=&\left|
\begin{array}{cc}
\partial_xZ_1 & \partial_yZ_1 \nonumber\\
\partial_xZ_2 & \partial_yZ_2
\end{array}
\right|=\left|
\begin{array}{cc}
a(t)+b(t)\partial_xR\rho_0 & b(t)\partial_yR\rho_0 \nonumber\\
-b\partial_xP\rho_0 & a(t)-b(t)\partial_yP\rho_0
\end{array}
\right|\\
=&\Big[a(t)+b(t)\partial_xR\rho_0\Big]^2+\Big[b(t)\partial_xP\rho_0\Big]^2\Big|_{(x,y)}>0.
\end{align}
Due to \eqref{eq:proper1} and \eqref{eq:proper2}, we obtain
\[
|Z(w,t)|\to+\infty~\textrm{ as }~|w|\to +\infty.
\]
which means $Z(\cdot,t)$ is proper \cite[Definition 6.2.2]{krantz2012implicit}.
By the Hadamard's global inverse function theorem \cite[Theorem 6.2.8]{krantz2012implicit}, there exists a inverse function $Z^{-1}(\cdot,t)$ such that
\[
Z^{-1}(\cdot,t):~\occ \to Z^{-1}(\occ,t)
\]
is a bijection. We also know $Z^{-1}$ is $\occ$-holomorphic since $Z$ is $\mathbb{C}_+$-holomorphic. Moreover, for any  $z\in \occ$, there exists $w=Z^{-1}(z,t)\in \mathbb{C}_+$. Due to $z=Z(Z^{-1}(z,t),t) \in \occ$  and $|Z_w(w,t)|\neq 0$ (by \eqref{eq:big0}), we have
\[
\partial_tZ^{-1}(z,t)=-\frac{\partial_tZ(w,t)}{\partial_wZ(w,t)},\quad w=Z^{-1}(z,t).
\]
Because of \eqref{eq:gloabltrajec}, we know $\frac{\partial^k}{\partial t^k}Z(w,t)$ is $\mathbb{C}_+$-holomorphic for any positive integer $k$.
Hence, $\frac{\partial^k}{\partial t^k}Z^{-1}(z,t)$ is $\occ$-holomorphic for any positive integer $k$.
From \eqref{eq:gloabltrajec}, we have
\begin{align}\label{eq:z}
z=Z^{-1}(z,t) \cosh t +g_0(Z^{-1}(z,t))\sinh t,~~z\in \occ.
\end{align}
By \eqref{eq:complexBurgerscharac}, we obtain
\[
g(Z(w,t),t)=\frac{\di }{\di t}Z(w,t)=w\sinh t+g_0(w)\cosh t.
\]
Hence,
\begin{align}\label{eq:g}
g(z,t)=Z^{-1}(z,t)\sinh t+g_0(Z^{-1}(z,t))\cosh t,
\end{align}
which is a $\occ$-holomorphic solution to the complex Burgers equation \eqref{eq:complexBurgers2} satisfying $g(z,0)=g_0(z)$. Moreover, due to the time regularity for $Z^{-1}(z,t)$, we know that $\frac{\partial^k}{\partial t^k}g(z,t)$ is $\occ$-holomorphic for any positive integer $k$ and $t>0$.

\textbf{Step 2.} Proof of (ii).
A $\occ$-holomorphic solution to \eqref{eq:complexBurgers1}  is given by
\begin{align}\label{eq:f}
f(z,t):=g(z,t)+ z,~~z\in\occ,~~t>0,
\end{align}
with initial datum $f_0(z)=\pi R\rho_0(x,y)-i\pi P\rho_0(x,y)$, $z=x+iy\in\mathbb{C}_+$.  Combining \eqref{eq:z} and \eqref{eq:g}, we obtain
\begin{align}\label{eq:fz}
f(z,t)= f_0(Z^{-1}(z,t))e^{t}~\textrm{ and }~z=e^{- t}Z^{-1}(z,t)+f_0(Z^{-1}(z,t))\sinh  t,~~z\in\occ.
\end{align}
Consider the trace of $f(z,t)$ on the real line and define:
\[
f(x,t)=:u(x,t)-i\pi \rho(x,t).
\]
Due to Lemma \ref{lmm:bijection}, for any $x\in\mathbb{R}$, we have $Z^{-1}(x,t)=:a_x+ib_x\in\mathbb{C_+}$ with some positive real number $b_x>0$. From \eqref{eq:fz}, we have
\[
f(x,t)= f_0(a_x+ib_x)e^{t}=\pi R\rho_0(a_x,b_x)e^t-i\pi P\rho_0(a_x,b_x)e^t
\]
Therefore,
\begin{align}
\rho(x,t)=P\rho_0(a_x,b_x)e^t>0,~~x\in\mathbb{R}.
\end{align}
Hence, $\rho(x,t)$ is a positive analytical solution to the Dyson equation \eqref{eq:meanfield}. Moreover, by the uniqueness of solutions to the characteristics equation \eqref{eq:complexBurgers2} we know analytical solutions to the Dyson equation \eqref{eq:meanfield} is unique.

The energy estimate \eqref{eq:energyconservated} follows from
\begin{align*}
\frac{\di}{\di t}E(\rho)&=\int_{\mathbb{R}}\frac{\delta E}{\delta \rho}\cdot \partial_t\rho \di x=-\int_{\mathbb{R}}\rho\left|\partial_x\left(\frac{\delta E}{\delta \rho}\right)\right|^2 \di x\nonumber\\
&=-\int_{\mathbb{R}}\rho(x,t)\big|\gamma x-\pi H\rho(x,t)\big|^2 \di x.
\end{align*}
For \eqref{eq:secondmoment}, direct calculations show that $\|\rho( t)\|_{L^1(\mathbb{R})}=\|\rho_0\|_{L^1(\mathbb{R})}$. Multiplying \eqref{eq:meanfield} by $x^2$ and taking integral yield
\[
\frac{\di}{\di t}\int_{\mathbb{R}}x^2\rho(x,t)\di x=2\pi\int_{\mathbb{R}}x\rho H\rho \di x-2\int_{\mathbb{R}}x^2\rho(x,t)\di x=\|\rho_0\|_{L^1}^2-2\int_{\mathbb{R}}x^2\rho(x,t)\di x,
\]
which implies \eqref{eq:secondmoment}.
Inequality \eqref{eq:entropy} follows from Gr\"onwall's inequality gives and the following estimate
\begin{align*}
\frac{\di}{\di t} \int_\mathbb{R} \rho \log \rho \di x =& \int_{\mathbb{R}} \partial_t \rho(\log \rho+1) \di x = \int_{\mathbb{R}} -(\rho H \rho+ \gamma x \rho)_x(\log \rho+1) \di x \\
=&  \int_{\mathbb{R}}(H \rho - \gamma x) \rho_x \di x  = - \|(-\Delta)^{1/4}\rho\|_{L^2}^2+\gamma\|\rho_0\|_{L^1}.
\end{align*}

\textbf{Step 3.} We prove  (iii) following the idea of \cite{rogers1993interacting}. Recall formula \eqref{eq:fz}.
For fixed $z\in\occ$, denote 
\[
z_r(t)+iz_i(t):=e^{-t}Z^{-1}(z,t).
\]
Next, we prove that $z_r(t)+iz_i(t)$ converges to a point $w=z_r^*+iz_i^*\in \mathbb{C}_+$ as $t\to\infty$. To this end, we first prove
$|z_r(t)|$ and $z_i(t)$ are all bounded from above and below uniformly in time $t$.

Because
\[
f_0(Z^{-1}(z,t))=\pi R\rho_0(e^tz_r(t),e^tz_i(t))-i\pi P\rho_0(e^tz_r(t),e^tz_i(t)),
\]
by \eqref{eq:fz}, we have
\begin{align}\label{eq:z1}
z=z_r(t)+\pi R\rho_0(e^tz_r(t),e^tz_i(t))\sinh t+i\left[z_i(t)-\pi P\rho_0(e^tz_r(t),e^tz_i(t))\sinh t\right].
\end{align}
Due to $\pi P\rho_0(e^tz_r(t),e^tz_i(t))\sinh t\geq 0$, we have
\[
z_i(t)\geq\Im(z)>0.
\]
Moreover, we have
\begin{align*}
\Im(z)&=z_i(t)-\pi P\rho_0(e^tz_r(t),e^tz_i(t))\sinh t\\
&=z_i(t)-\int_{\mathbb{R}}\frac{e^tz_i(t)}{e^{2t}z_i^2(t)+(e^tz_r(t)-s)^2}\rho_0(s)\di s\sinh t\\
&\geq z_i(t)-\int_{\mathbb{R}}\frac{e^{2t}z_i(t)}{2e^{2t}z_i^2(t)+2(e^tz_r(t)-s)^2}\rho_0(s)\di s\\
&\geq z_i(t)-\frac{1}{z_i(t)},
\end{align*}
which implies
\[
z_i(t)\leq \Im(z)+1.
\]
Hence, $z_i(t)$ is bounded as
\[
0<\Im(z)\leq z_i(t)\leq \Im(z)+1
\]
Next, we prove
\[
\sup_{t\geq0}|z_r(t)|<+\infty.
\]
We prove this by a contradiction argument. If there exists $t_n\to\infty$ such that $z_r(t_n)\to\infty$, then by the dominated convergence theorem we have
\[
\pi R\rho_0(e^{t_n}z_r(t_n),e^{t_n}z_i(t_n))\sinh t_n=\int_{\mathbb{R}}\frac{e^{t_n}z_r(t_n)-s}{e^{2t_n}z_i^2(t_n)+(e^{t_n}z_r(t_n)-s)^2}\di x\sinh t_n\to0,~~n\to\infty.
\]
By \eqref{eq:z1}, we obtain a contradiction that
\[
\Re(z)=z_r(t_n)+\pi R\rho_0(e^{t_n}z_r(t_n),e^{t_n}z_i(t_n))\sinh t_n\to\infty.
\]

Since  $|z_r(t)|$ and $z_i(t)$ are bounded, there exist $t_n\to\infty$  and two constant $z_r^*$, $z_i^*>0$ such that
\[
z_r(t_n)\to z_r^*,~~z_i(t_n)\to z_i^*,~~n\to\infty.
\]
Let $w:=z_r^*+iz_i^*$. For any $s\in\mathbb{R}$, we have
\[
\frac{e^{t_n}z_r(t_n)-s}{e^{2t_n}z_i^2(t_n)+(e^{t_n}z_r(t_n)-s)^2}\sinh t_n\to\frac{z_r^*}{2(z_i^*)^2+2(z_r^*)^2},~~n\to\infty.
\]
Then, by the dominated convergence theorem we have
\begin{align*}
&\lim_{n\to\infty}\pi R\rho_0(e^tz_r(t_n),e^tz_i(t_n))\sinh t_n\\
=&\lim_{n\to\infty}\int_{\mathbb{R}}\frac{e^{t_n}z_r(t_n)-s}{e^{2t_n}z_i^2(t_n)+(e^{t_n}z_r(t_n)-s)^2}\di x\sinh t_n\\
=&\frac{z_r^*}{2(z_i^*)^2+2(z_r^*)^2}.
\end{align*}
Similarly, we have
\begin{align*}
\lim_{n\to\infty}\pi P\rho_0(e^tz_r(t_n),e^tz_i(t_n))\sinh t=\frac{z_i^*}{2(z_i^*)^2+2(z_r^*)^2}.
\end{align*}
Hence, from \eqref{eq:z1} we obtain
\[
z=w+\frac{1}{2}\frac{z_r^*-iz_i^*}{(z_i^*)^2+(z_r^*)^2}=w+\frac{1}{2w}.
\]
Similar to the calculation of \eqref{eq:imaginary}, we know that the above equation has a unique solution in $\mathbb{C_+}$:
\[
w=\frac{1}{z-\sqrt{z^2-2}}.
\]
Hence, we have
\[
e^{-t}Z^{-1}(z,t)=z_r(t)+iz_i(t)\to w=\frac{1}{z-\sqrt{z^2-2}},~~t\to\infty.
\]
By \eqref{eq:fz} and using the dominated convergence theorem again, we have
\begin{align*}
f(z,t)&=f_0(Z^{-1}(z,t))e^t\\
&=\int_{\mathbb{R}}\frac{e^{2t}z_r(t)-s}{e^{2t}z_i^2(t)+[e^tz_r(t)-s]^2}\rho_0(s)\di s-i\int_{\mathbb{R}}\frac{e^{2t}z_i(t)}{e^{2t}z_i^2(t)+[e^tz_r(t)-s]^2}\rho_0(s)\di s\\
&\to \frac{z_r^*-iz_i^*}{(z_i^*)^2+(z_r^*)^2}=\frac{1}{w}=z-\sqrt{z^2-2}.
\end{align*}
The trace of $z-\sqrt{z^2-2}$ on the real line is
\begin{gather*}
f_\infty(x):=\pi H\rho_\infty(x)-i\pi\rho_\infty(x)=\left\{
\begin{split}
&x+\sqrt{x^2-2},~~x<-\sqrt{2},\\
&x-i\sqrt{2-x^2},~~x\in[-\sqrt{2},\sqrt{2}],\\
&x-\sqrt{x^2-2} ,~~x>\sqrt{2},
\end{split}
\right.
\end{gather*}
Hence,
\[
\rho(x,t)\to \rho_\infty(x)=\frac{\sqrt{(2-x^2)_+}}{\pi},
\]
which proves part (iii) in Theorem \ref{thm:analytic}.

\textbf{Step 4.} We prove  (iv). From \eqref{eq:transform}, if $g(z,t)$ and $\rho(x,t)$ are analytical solutions to \eqref{eq:complexBurgers3} and \eqref{eq:meanfield} with $\gamma=0$, then
\[
\tilde{g}(z,t):=e^{t} g\left(e^{t}z,\frac{e^{2t-1}}{2}\right)- z,~~z\in\occ,~~t>0,
\]
is a $\occ$-holomorphic solution to \eqref{eq:complexBurgers2},
and
\[
\tilde{\rho}(x,t):=e^{t} \rho\left(e^{t} x, \frac{e^{2 t}-1}{2}\right),~~x\in\mathbb{R},~~t>0,
\]
gives an analytical solution to \eqref{eq:meanfield} for $\gamma=1$. By part (iii), we obtain part (iv).

\end{proof}

\section{An explicit solution to \eqref{eq:meanfield} with $\gamma=0$}\label{app_A}
In this section, by the Stieltjes transform of Wigner's semicircle law $\mu_1$ in \eqref{eq:semicirclelaw},  we recover an explicit solution, which is same as the explicit solution of the Dyson equation \eqref{eq:meanfield} constructed in \eqref{eq:explicitsolution} and \eqref{app7}  (see  \eqref{eq:Lowurho}). 

First, we begin by taking Stieltjes transformation.
 Let $f_1(z)$ be the Stieltjes transform of the Wigner's semicircle law $\mu_1$ given by \eqref{eq:semicirclelaw}:
\[
f_1(z)=\int_{-2}^2 \frac{1}{z-y} \mu_1(\di y),~~z\in\mathbb{C} \setminus[-2,2].
\]
Let $y=2\cos\theta$ for $\theta\in[-\pi,0]$, $\alpha=-\theta$ and we have
\begin{align*}
f_1(z)=\frac{1}{2\pi}\int_{-2}^2\frac{\sqrt{4-y^2}}{z-y}\di y=&\frac{1}{\pi}\int^0_{-\pi} \frac{2\sin^2\theta}{z-2\cos\theta}\di \theta=\frac{1}{\pi}\int^\pi_{0} \frac{2\sin^2\alpha}{z-2\cos\alpha}\di \alpha\\
=&\frac{1}{\pi}\int^{\pi}_{-\pi} \frac{\sin^2\theta}{z-2\cos\theta}\di \theta.
\end{align*}
Let $\zeta=e^{i\theta}$ and we  obtain
\[
f_1(z)=\frac{1}{4\pi i}\oint_{|\zeta|=1}\frac{(\zeta^2-1)^2}{\zeta^2(\zeta^2+1-z\zeta)}\di\zeta.
\]
Set
\[
h(\zeta):=\frac{(\zeta^2-1)^2}{\zeta^2(\zeta^2+1-z\zeta)}.
\]
Function $h(\zeta)$ has three poles: $\zeta_0=0$, $\zeta_1=\frac{z+\sqrt{z^2-4}}{2}$, and $\zeta_2=\frac{z-\sqrt{z^2-4}}{2}$.
Next, we choose the branch cut of $\sqrt{z^2-4}$. Due to
\[
\sqrt{z^2-4}=|z^2-4|^{1/2}e^{\frac{i}{2}[\arg(z-2)+\arg(z+2)]},
\]
we see that $-2$ and $2$ are branch points. We take the branch cut along the interval $[-2,2]$ and we set $\arg(z-2)=\pi$ and $\arg(z+2)=0$ for $z$ on the upside of the branch cut. In this case, on the upside of $[-2,2]$  we have $\sqrt{z^2-4}=i\sqrt{4-x^2}$ while on the downside of $[-2,2]$, $\sqrt{z^2-4}=-i\sqrt{4-x^2}$ . Moreover, the square root of $z^2-4$ has a positive imaginary part when $z\in \mathbb{C}_+$ and it has a negative imaginary part when $z\in \mathbb{C}_-:=\{z:\Im(z)<0\}$. Hence, for the imaginary part, we have
\begin{align}\label{eq:zeta12}
|\Im(z-\sqrt{z^2-4})|<|\Im(z+\sqrt{z^2-4})|~\textrm{ for }~z\in\mathbb{C}\setminus[-2,2].
\end{align}
which implies
\[
|\Im(\zeta_2)|<|\Im(\zeta_1)|~\textrm{ for }~z\in\mathbb{C}\setminus[-2,2].
\]
Due to $\zeta_1\zeta_2=1$, we have
\[
|\zeta_2|<1 \text{ and }|\zeta_1|>1~\textrm{ for }~z\in\mathbb{C}\setminus[-2,2].
\]
And we obtain for $z\in\mathbb{C}\setminus[-2,2],$
\[
\Res h(\zeta_0)=\lim_{\zeta\to \zeta_0}\frac{\di }{\di \zeta}[(\zeta-\zeta_0)^2h(\zeta)]=z, \quad 
\Res h(\zeta_2)=-\sqrt{z^2-4}
\]
Hence, by the Residue theorem,
\begin{align}\label{eq:Deffstar}
f_1(z)=\frac{z-\sqrt{z^2-4}}{2},\qquad \mathbb{C}\setminus[-2,2].
\end{align}

Second, we show $f_1(-z)$ is a Herglotz  (Pick) function, which is analytical on $\mathbb{C}\setminus [-2, 2]$ and $\Im(z)\Im (f_1(-z)) > 0$ for $\Im(z)\neq0$ and we  show the decay order of $\Re(f_1)$ and  $\Im(f_1)$  as $\Re(z)$ and $\Im(z)$ tends to infinity. 
Direct calculations give
\begin{align}\label{eq:relations}
2\Re(\sqrt{z})\Im(\sqrt{z})=\Im(z),~~\Re(\sqrt{z})^2=\frac{|z|+\Re(z)}{2},~\textrm{ and }~ \Im(\sqrt{z})^2=\frac{|z|-\Re(z)}{2}.
\end{align}
For $z=x+iy$,  we obtain
\[
\Im(z^2-4)=2xy,\quad \Re(\sqrt{z^2-4})^2=\frac{\sqrt{(x^2-y^2-4)^2+4x^2y^2}+x^2-y^2-4}{2},
\]
and
\[
\Im(\sqrt{z^2-4})^2=\frac{\sqrt{(x^2-y^2-4)^2+4x^2y^2}-(x^2-y^2-4)}{2}.
\]
Recall that in our settings of branch cut, the square root of $z^2-4$ has positive imaginary part when $z\in \mathbb{C}_+$ and it has negative imaginary part when $z\in \mathbb{C}_-$. Due to  \eqref{eq:relations}, we know that the sign of $\Re(\sqrt{z^2-4})$ is the same as $\Im(\sqrt{z^2-4})$ when $xy>0$ and they have different signs if $xy<0$.
By elementary calculations, we have
\begin{gather}\label{tmb4}
\Re(f_1(z))=\left\{
\begin{split}
\frac{\sqrt{2}x-\sqrt{\sqrt{(x^2-y^2-4)^2+4x^2y^2}+(x^2-y^2-4)}}{2\sqrt{2}},~~x>0,\\
\frac{\sqrt{2}x+\sqrt{\sqrt{(x^2-y^2-4)^2+4x^2y^2}+(x^2-y^2-4)}}{2\sqrt{2}},~~x<0,
\end{split}
\right.
\end{gather}
and
\begin{gather}\label{eq:imaginary}
\Im(f_1(z))=\left\{
\begin{split}
\frac{\sqrt{2}y-\sqrt{\sqrt{(x^2-y^2-4)^2+4x^2y^2}-(x^2-y^2-4)}}{2\sqrt{2}}<0,~~y>0,\\
\frac{\sqrt{2}y+\sqrt{\sqrt{(x^2-y^2-4)^2+4x^2y^2}-(x^2-y^2-4)}}{2\sqrt{2}}>0,~~y<0.
\end{split}
\right.
\end{gather}
From the sign in \eqref{eq:imaginary},  $\Im (z)>0$ implies $\Im(-z)<0$ and  $\Im (f(-z))>0$. Therefore we have $\Im(z)\cdot \Im(f_1(-z))\ge0$ and thus $f_1(-z)$ is a Herglotz function.
Moreover, for fixed $y\in\mathbb{R}$ in \eqref{tmb4}, dividing $\Re(f_1(z))$  by $x$   shows that $\Re(f_1(z))$   decays in the order $O(|x|^{-1})$ as $|x|\to\infty$. Similarly, $\Im(f_1(z))$ decays in the order $O(|y|^{-1})$ as $|y|\to\infty$ for fixed $x\in\mathbb{R}$.

Third, we use $f_1$ to recover the explicit solution to the Dyson equation \eqref{eq:meanfield} with $\gamma=0$ given by \eqref{eq:explicitsolution} and \eqref{app7}. Define
\[
f_t(z)=\frac{1}{\sqrt{t}}f_1\left(\frac{z}{\sqrt{t}}\right),~~z\in \mathbb{C}\setminus [-2\sqrt{t},2\sqrt{t}].
\]
Then, direct checking shows that $f_t(z)$ is a self-similar solution to complex Burgers equation \eqref{eq:complexBurgers}.

Finally, we try to obtain the traces of $f_1$ on the upper and lower half planes respectively. In the above settings of branch cut,  we have
\[
\arg(z-2)=\pi=\arg(z+2),~~z\in(-\infty,-2),
\]
which implies
\[
\sqrt{z^2-4}=\sqrt{x^2-4}e^{i\pi}=-\sqrt{x^2-4},~~z=x\in(-\infty,-2).
\]
Similarly, we have
\[
\sqrt{z^2-4}=\sqrt{x^2-4}e^{i2\pi}=\sqrt{x^2-4},~~z=x\in(2,+\infty).
\]
Hence, the trace of function $f_1(z)$ defined by \eqref{eq:Deffstar} from the upper half plane $\mathbb{C}_+$ is  given by
\begin{gather}\label{eq:Uperlimits}
f_1(x+)=\left\{
\begin{split}
&\frac{x+\sqrt{x^2-4}}{2},~~x<-2,\\
&\frac{x-i\sqrt{4-x^2}}{2},~~x\in [-2,2],\\
&\frac{x-\sqrt{x^2-4}}{2},~~x>2.
\end{split}
\right.
\end{gather}
The trace of function $f_1(z)$ given by \eqref{eq:Deffstar} from the lower half plane $\mathbb{C}_-$ is
\begin{gather}\label{eq:Lowlimits}
f_1(x-)=\left\{
\begin{split}
&\frac{x+\sqrt{x^2-4}}{2},~~x<-2,\\
&\frac{x+i\sqrt{4-x^2}}{2},~~x\in [-2,2],\\
&\frac{x-\sqrt{x^2-4}}{2},~~x>2.
\end{split}
\right.
\end{gather}
Direct computations show that $\frac{1}{\sqrt{t}}f_1(\frac{x}{\sqrt{t}}\pm)$ are solutions to complex Burgers equation on the real line $\mathbb{R}$.

Recall Section \ref{sec:Dyson}. If $\rho$ is a solution to the Dyson equation  \eqref{eq:meanfield}, then $g=\pi H\rho-i\pi \rho-x$ is a solution to the complex Burgers equation \eqref{eq:complexBurgers} on the real line and $f=\pi H\rho-i\pi \rho$ gives a trace of an analytical function on the upper half plane.
Hence, we use the trace $f_1(x+)$ (given by \eqref{eq:Uperlimits}) to define
\begin{gather*}
f(x,t)=\frac{1}{\sqrt{t}}f_1\left(\frac{x}{\sqrt{t}}+\right)=\left\{
\begin{split}
&\frac{x}{2t}+\frac{\sqrt{x^2-4t}}{2t},~~x<-2\sqrt{t},\\
&\frac{x}{2t}-i\frac{\sqrt{4t-x^2}}{2t},~~x\in[-2\sqrt{t},2\sqrt{t}],\\
&\frac{x}{2t}-\frac{\sqrt{x^2-4t}}{2t},~~x>2\sqrt{t},
\end{split}
\right.
\end{gather*}
and
\begin{gather}\label{eq:Lowurho}
u(x,t)=\left\{
\begin{split}
&\frac{x+\sqrt{x^2-4t}}{2t},~~x<-2\sqrt{t},\\
&\frac{x}{2t},~~x\in[-2\sqrt{t},2\sqrt{t}],\\
&\frac{x-\sqrt{x^2-4t}}{2t},~~x>2\sqrt{t},
\end{split}
\right.\quad
\rho(x,t)=\frac{\sqrt{(4t-x^2)_+}}{2\pi t}.
\end{gather}

To the end of this section, we provide another method to prove $f_1(-z)$ is a Herglotz (Pick) analytic on $(-\infty,-2)\cup (2, +\infty)$. Recall $\mu_1(\di y) = \frac{1}{2\pi }\sqrt{(4-y^2)_+} \di y$. Then changing of variable $y=t-2$ gives that
\begin{equation}
f_1(-z) = \int_0^4 \frac{1}{-z+2 -t} \frac{1}{2\pi} \sqrt{(t(4-t))_+} \di t.
\end{equation}
Define the measure $\di \mu_{*}(t):= \frac{1}{2\pi} \sqrt{(t(4-t))_+} \di t$ and recast $f_1(-z)$ as
\begin{equation}
f_1(-z)= \int_0^4 \frac{1}{-z+2 - t} \di \mu_*(t).
\end{equation} 
Changing variable $-z+2 = \frac{1}{w}$ gives
\begin{equation}
f_1(-z)= \int_0^4 \frac{1}{\frac{1}{w} - t} \di \mu_*(t)= \int_0^4 \frac{w}{1-wt} \di \mu_*(t)=: wF(w).
\end{equation}
Here $F(w)= \int_0^4 \frac{1}{1-wt} \di \mu_*(t) $ is a Pick function analytic on $(-\infty, \frac14)$. $F(w)$ is also the generating function of a completely monotone sequence
$\{A_n(2,2)\}_{n\geq 0}$ \cite[Lemma 3]{liupego}, where $A_n(2,2)$ is the general Fuss-Catalan numbers (also called Raney numbers) with index $(2,2)$. Therefore from \cite[Corollary 1 (iii)]{liupego}, $F_1(w):=wF(w)$ is a Pick function analytic on $w\in(-\infty, \frac14)$. From the relation $-z+2 = \frac{1}{w}$, we know $w(z)=\frac{1}{2-w}$ is a pick function mapping $(-\infty,-2)\cup (2, +\infty)$ to $(-\infty, \frac14)$.    Therefore the composition $f_1(-z)=F_1\circ w(z)$ is a Pick function analytic on $z\in(-\infty,-2)\cup (2, +\infty)$.

\end{document}